\documentclass[11pt]{article}
\usepackage[title]{appendix}
\usepackage[latin1]{inputenc}
\usepackage{subfigure}
\usepackage{color}
\usepackage{amsmath}
\usepackage{amsthm}
\usepackage{verbatim}
\usepackage{titling}
\usepackage{amssymb}

\usepackage{amsfonts}
\usepackage{graphicx}
\usepackage{mathrsfs}
\usepackage{bbm}
\usepackage[all]{xy}
\usepackage{mathtools}
\usepackage{tabularx}
\usepackage{array}
\usepackage{enumitem}
\usepackage{listings}
\usepackage{url}
\theoremstyle{plain}

\newtheorem{theorem}{Theorem}[section]
\newtheorem{conjecture}[theorem]{Conjecture}

\newtheorem{lemma}[theorem]{Lemma}

\theoremstyle{definition}

\newtheorem{example}[theorem]{Example}

\DeclareMathAlphabet{\mathpzc}{OT1}{pzc}{m}{it}

\usepackage{amssymb,amsmath,tabularx,graphicx}
\usepackage{tikz}
\usetikzlibrary[calc,intersections,through,backgrounds,arrows,decorations.pathmorphing]

\begin{document}
	
	\title{A non-uniform extension of Baranyai's Theorem}
	\author{Jinye He \thanks{School of Mathematical Sciences, University of Science and Technology of China, Hefei, Anhui 230026, China. Email: zykxxxx@mail.ustc.edu.cn.}
    \and Hao Huang \thanks{Department of Mathematics, National University of Singapore. Research supported in part by a start-up grant at NUS and an MOE Academic Research Fund (AcRF) Tier 1 grant. Email: huanghao@nus.edu.sg.}
	\and Jie Ma \thanks{School of Mathematical Sciences, University of Science and Technology of China, Hefei, Anhui 230026,
		China. Research supported in part by the National Key R and D Program of China 2020YFA0713100, National Natural Science Foundation of China grant 12125106,
Innovation Program for Quantum Science and Technology 2021ZD0302904, and Anhui Initiative in Quantum Information Technologies grant AHY150200. Email: jiema@ustc.edu.cn.}}
	\date{}

\maketitle
\begin{abstract}
A celebrated theorem of Baranyai states that when $k$ divides $n$, the family $K_n^k$ of all $k$-subsets of an $n$-element set can be partitioned into perfect matchings. In other words, $K_n^k$ is $1$-factorable. In this paper, we determine all $n, k$, such that the family $K_n^{\le k}$ consisting of subsets of $[n]$ of size up to $k$ is $1$-factorable, and thus extend Baranyai's Theorem to the non-uniform setting.
%establish a non-uniform analogue of Baranyai's Theorem.
In particular, our result implies that for fixed $k$ and sufficiently large $n$, $K_n^{\le k}$ is $1$-factorable if and only if $n \equiv 0$ or $-1 \pmod k$.
\end{abstract}

\section{Introduction}\label{sec_intro}
A \textit{hypergraph} is a system of subsets of finite sets. Formally, a hypergraph $H=(V, E)$ consists of a vertex set $V$, and an edge set $E$ which is a family of non-empty subsets of $V$. A \textit{$k$-uniform hypergraph} is a hypergraph such that all its edges are of size $k$.  A \textit{$\ell$-factor} of a hypergraph $H$ is a spanning sub-hypergraph $H'$ in which every vertex is contained in $\ell$ edges. We say that $H$ has a \textit{$\ell$-factorization} if its edge set can be partitioned into $\ell$-factors. A hypergraph $H$ is said to be \textit{$\ell$-factorable} if it admits a $\ell$-factorization.
There have been extensive research on $1$-factorization of graphs (see \cite{AF,CH85,CH89,FJ18,FJS,H81,MR,PR97,PT01,SSTF,W92} and the resolution of the $1$-factorization conjecture \cite{CKLOT}).

We denote the complete $k$-uniform hypergraph on $n$ vertices by $K_n^k$. Clearly, a necessary condition  for $K_{n}^k$ to be $1$-factorable is $k\mathrel{\mid}n$. It turns out that this is also sufficient for $k=2$ (folklore) and $k=3$ (proved by Peltesohn \cite{Peltesohn} in 1936). The sufficiency for general $k$ was eventually established by Baranyai \cite{Baranyai} in 1975 as follows.

\begin{theorem}[Baranyai \cite{Baranyai}]
	For any positive integers $k, n$ such that $k$ divides $n$, the complete $k$-uniform hypergraph $K_n^k$ can be decomposed into ${\binom {n}{k}}{\frac {k}{n}}={\binom {n-1}{k-1}}$ $1$-factors.
\end{theorem}

His proof was based on an ingenious use of the Max-Flow Min-Cut Theorem.
Generalizations and extensions of Baranyai's Theorem can be found in \cite{Ba14,Ba17,BJN,BN18,BB77,Ca76,HH93}.

In this paper, we mainly consider the non-uniform hypergraph $K_n^{\le k}$, whose vertex set is $[n]=\{1, \cdots, n\}$ and edge set consists of all the non-empty subsets of $[n]$ of size up to $k$, denoted by $\binom{[n]}{\le k}$. We show that when $k$ is fixed and $n$ is sufficiently large, a necessary and sufficient condition for such hypergraph to be $1$-factorable is that $n$ is congruent to $0$ or $-1$ modulo $k$. Our result is actually much more precise.

\begin{theorem}\label{thm_main}
For positive integers $n, k$ such that $k<n/2$, $K_n^{\le k}$ is $1$-factorable if and only if one of the two following conditions is met:\\
(i)  $n\equiv 0 \pmod k$ and $n\ge k(k-2)$,\\
(ii) $n\equiv -1 \pmod k$ and $n\ge k(\lceil \frac{k}{2}\rceil -1)-1$.
\end{theorem}

The $k \ge n/2$ case can be reduced to the previous range by the following equivalence.

\begin{theorem}\label{thm_equiv}
For positive integers $n, k$ such that $n/2 \le k \le n-1$, $K_{n}^{\le k}$ is $1$-factorable if and only if $K_n^{\le n-k-1}$ is $1$-factorable.
\end{theorem}

Theorems \ref{thm_main} and \ref{thm_equiv} together provide a complete characterization of all $n, k$ such that $K_{n}^{\le k}$ has a $1$-factorization.

The rest of this paper is organized as follows: in the next section, using the Max-Flow Min-Cut Theorem, we show that the $1$-factorization problem is equivalent to finding non-negative integer solutions to a system of linear equations given by the partitions of $[n]$ into parts of size up to $k$. Section \ref{sec_solve} determines when this equivalent problem is feasible. Section \ref{sec_multilayer} discusses an extension of Theorem \ref{thm_main} to other families of subsets obtained from taking the union of multiple levels of the hypercube. The last section contains some concluding remarks and open problems.

\section{A reduction using network flow}\label{sec_flow}
In this section, we reduce the $1$-factorization problem of $K_{n}^{\le k}$ to finding non-negative integer solutions to a system of linear equations (as we should see soon, both problems in fact are equivalent).
Throughout this section, let $n,k$ be two fixed positive integers with $n\ge k$,
and let $\mathcal{L}$ be a set consisting of $k$ and some positive integers in $\{1,2,\cdots,k-1\}$.
We denote by $\binom{[n]}{\mathcal{L}}$ the family of subsets of $[n]$ whose size is an element of $\mathcal{L}$.
We will prove a more general reduction for the $1$-factorization of $\binom{[n]}{\mathcal{L}}$ which holds for any $\mathcal{L}$ (the case when $\mathcal{L}=[k]$ corresponds to our main results).

For given $n,k$ and $\mathcal{L}$, an {\it $(n, \mathcal{L})$-type}, or simply a {\it type} is a vector $\vec{\lambda}=(\lambda_1,\lambda_2,\cdots,\lambda_k)$ in $\mathbb {Z}_{\ge 0}^{k}$
such that $\sum_{j\in \mathcal{L}}j\cdot \lambda_j=n$ and $\lambda_j=0$ for every $j\in [k]\backslash\mathcal{L}$.
If $\mathcal{L}=[k]$, then we also call it an {\it $(n, k)$-type}.
Let $|\vec{\lambda}|=\sum_{j=1}^{k}\lambda_k$.	
We say that $\mathscr{A}=\{A_1,A_2,\cdots,A_t\}$ is a \textit{$\vec{\lambda}$-partition} of $[n]$, if $A_i$'s are pairwise disjoint subsets of $[n]$ such that $\cup_{i=1}^t A_i =[n]$, and for every $1 \le j \le k$, there are $\lambda_j$ subsets $A_i$'s of size $j$.

For given $n,k$, we now define a matrix $A_\mathcal{L}$, which we will soon show to be closely related to the $1$-factorization of $\binom{[n]}{\mathcal{L}}$.
Let $A_\mathcal{L}$ be the matrix with $k$ columns composed by all admissible $(n, \mathcal{L})$-types $\vec{\lambda}$ as follows:
	$$A_\mathcal{L}= \begin{pmatrix}
		\vec{\lambda}\\
		\vec{\lambda}'\\
		\vdots
	\end{pmatrix}$$	
When $\mathcal{L}$ is clear from the context, we will simply write it for $A$. Below we give an example of $A_\mathcal{L}$.
\begin{example}
	For $n=7$, $k=3$ and $\mathcal{L}=[3]$, the matrix $A$ is defined as follows. For example, the third row corresponds to $\vec{\lambda}=(2, 1, 1)$ which is a $(7, 3)$-type because $7=(2,1,1)\cdot (1,2,3)^T$.
	\begin{align}
		&A_\mathcal{L}=
		\begin{pmatrix}
			1 & 0 & 2 \\
			0 & 2 & 1 \\
			2 & 1 & 1 \\
			4 & 0 & 1 \\
			1 & 3 & 0 \\
			3 & 2 & 0 \\
			5 & 1 & 0 \\
			7 & 0 & 0
		\end{pmatrix}\nonumber
	\end{align}
\end{example}

For given $n, k$ and $\mathcal{L}$, suppose the hypergraph $\binom{[n]}{\mathcal{L}}$ can be decomposed into $1$-factors $\mathscr{A}_1, \cdots, \mathscr{A}_m$.\footnote{It is easy to see that we must have $m=\sum_{j\in \mathcal{L}} \binom{n-1}{j-1}.$}
Denote the number of $\mathscr{A}_j$'s that are $\vec{\lambda}$-partitions of $[n]$ by $x_{\vec{\lambda}}$. Then the total number of $i$-subsets of $[n]$ appeared in $\{\mathscr{A}_1, \cdots, \mathscr{A}_m\}$ is equal to the $i$-th coordinate of $(A_\mathcal{L})^T \vec{x}$, where $\vec{x}=(x_{\vec{\lambda}})$ is a column vector indexed by all the $(n,\mathcal{L})$-types $\vec{\lambda}$.
Observe that $\binom{[n]}{\mathcal{L}}$ contains exactly $\binom{n}{i}$ subsets of size $i$ for each $i\in \mathcal{L}$.
Thus by setting $\vec{b}_\mathcal{L}=(b_1, \cdots, b_k)^T$ with $b_i=\binom{n}{i}$ for $i\in \mathcal{L}$ and $b_i=0$ for $i\in [k]\backslash \mathcal{L}$, we immediately have
\begin{equation}\label{eq1.1}
	(A_\mathcal{L})^T \vec{x}=\vec{b}_\mathcal{L}.
\end{equation}
Therefore the system \eqref{eq1.1} having a non-negative integer solution (meaning that all $x_{\vec{\lambda}}$ are non-negative integers) is a necessary condition for $\binom{[n]}{\mathcal{L}}$ to be $1$-factorable.

Our next theorem shows that this is indeed sufficient.
We remark that for Baranyai's Theorem (corresponding to $n,k$ and $\mathcal{L}=\{k\}$ with $k\mid n$), a non-negative integer solution to the corresponding system exists trivially because the only type involved is $(0, \cdots, 0, n/k)$ and $\frac{n}{k} \mid \binom{n}{k}$.
\begin{theorem}\label{thm_reduction}
Given positive integers $n, k$ and a set $\mathcal{L}$ of positive integers with $n\ge k$ and $k\in \mathcal{L}\subseteq [k]$,
the hypergraph $\binom{[n]}{\mathcal{L}}$ is $1$-factorable if and only if the system of linear equations \eqref{eq1.1} associated with it has a non-negative integer solution.
\end{theorem}

We will imitate Baranyai's ideas and construct a flow network to prove Theorem \ref{thm_reduction}.
Here we give a brief review of the definition of flow network and the statement of the Max-Flow Min-Cut Theorem of Ford and Fulkerson (1956),
to facilitate our later discussion. 
A {\it network} is a finite digraph $D = (V,E)$ together with two distinguished vertices called the {\it source} $s$ and the {\it sink} $t$, and a {\it capacity function} $\kappa : E(D) \rightarrow \mathbb{R}_{\ge 0}$ which associates a non-negative real number $\kappa(a)$ to each arc $a \in E(D)$.
The source $s$ must be the tail of every arc containing $s$, and the sink $t$ must be the head of every arc containing $t$.
We further assume that $D$ does not contain any arc of the form $a = (v, v)$ for a vertex $v \in V$. 
A {\it flow} on $D$ is a function $f: E(D) \rightarrow \mathbb{R}_{\ge 0}$ which assigns to each arc $a \in E(D)$ a non-negative real number $f(a)$ such that
\begin{enumerate}
	\item (Capacity Constraint) $0 \le f(a) \le \kappa(a)$ for all arcs $a \in E(D)$;
	\item (Conservation of Flow) for each vertex $v\in V\setminus \{s, t\}$, we have$$\sum\limits_{a\in E(D):~\textrm{$v$ is the head of $a$}}f(a) = \sum\limits_{a\in E(D):~\textrm{$v$ is the tail of $a$}}f(a).$$
\end{enumerate}
The {\it value} of the flow $f$, denoted by $|f|$, is the sum of $f(a)$ over all arcs $a$ leaving $s$.
A {\it cut} $(S,T)$ is a partition of $V(D)=S \cup T$ such that $s \in S$ and $t \in T$. 
The {\it capacity} of a cut $(S,T)$, denoted by $c(S,T)$, is the sum of the capacities of the arcs which has tail in $S$ and head in $T$, 
that is $$c(S,T)=\sum_{xy \in E(D):~x \in S, ~ y \in T} \kappa(xy).$$

\begin{theorem}[The Max-Flow Min-Cut Theorem]
Given a flow network $D = (V,E)$, the maximum value of a flow on $D$ is equal to the minimum capacity over all cuts in $D$.
\end{theorem}
We will also utilize the Integral Flow Theorem by Dantzig and Fulkerson \cite{DF}.
\begin{theorem}[The Integral Flow Theorem]\label{DF}
 If $D = (V,E)$ is a network in which every arc has integral capacity, then there exists a maximum flow $f$ on $D$ such that for each $a\in E(D)$, $f(a)$ is an integer.
\end{theorem}
Now we are ready to prove Theorem \ref{thm_reduction}.
In the coming proofs, for any integers $a,b$, the binomial coefficient $\binom{a}{b}$ is interpreted as zero whenever $a<0$, $b<0$ or $a<b$.
In particular, we have $\binom{0}{b}=1$ if $b=0$ and $\binom{0}{b}=0$ otherwise.
In this way, $\binom{a}{b}=\binom{a-1}{b-1}+\binom{a-1}{b}$ holds for any integers $a,b$.

\bigskip

\noindent {\bf Proof of Theorem \ref{thm_reduction}}:
Throughout this proof, $n, k$ and $\mathcal{L}$ are fixed.
From our previous discussions, it suffices to show that if the system \eqref{eq1.1} has a non-negative integer solution $\vec{x}=(x_{\vec{\lambda}})$ where $\vec{\lambda}$ is over all $(n,\mathcal{L})$-types,
then $\binom{[n]}{\mathcal{L}}$ is $1$-factorable.

For a given $(n, \mathcal{L})$-type $\vec{\lambda}$, we slightly extend the definition of a $\vec{\lambda}$-partition of $[n]$ to partitions of $[\ell]$ for any $0 \le \ell \le n$.
A partition $\mathscr{A}=\{A_1,A_2,\dots,A_t\}$ of $[\ell]$ with $t=|\vec{\lambda}|$ is called a {\it $\vec{\lambda}$-partition} of $[\ell]$,
if for every $j\in \mathcal{L}$, we assign the label, which we call {\it potential value}, $j$ to exactly $\lambda_j$ subsets $A_i$.
We point out here that repetitions are allowed only for empty sets.
Let $$M=\sum_{j\in \mathcal{L}}\binom{n-1}{j-1}.$$
We will prove the following statement by induction on $\ell$: for any $0\le \ell\le n$, there exists a collection of $\mathscr{A}_1,\mathscr{A}_2,\dots, \mathscr{A}_M$ of partitions
of $[\ell]$ such that all of the following hold:
	\begin{enumerate}
        \item[(1)] each set $S$ appeared in each partition is associated with a potential value $j\in \mathcal{L}$ with $j\ge |S|$,
		\item[(2)] for each $(n, \mathcal{L})$-type $\vec{\lambda}$, there are exactly $x_{\vec{\lambda}}$ partitions $\mathscr{A}_i$ that are $\vec{\lambda}$-partitions, and
		\item[(3)] for each $S\subseteq [\ell]$ and each $j\in \mathcal{L}$ with $j\geq |S|$, $S$ occurs $\binom{n-\ell}{j-|S|}$ times with potential value $j$ as subsets in the partitions $\mathscr{A}_1, \cdots, \mathscr{A}_M$.
	\end{enumerate}
Observe that when $\ell=n$, the third property ensures that every set $S$ in $\binom{[n]}{\mathcal{L}}$ appears exactly once (only with potential value $j=|S|$), which would provide a 1-factorization of $\binom{[n]}{\mathcal{L}}$.
It would be helpful to view this inductive proof as an evolution where each set (say with potential value $j$) in the partitions grows from an empty set to a set of size $j$ gradually.
	
Now we start the proof. For the base case when $\ell=0$, the existence of $\{\mathscr{A}_1, \cdots, \mathscr{A}_M\}$ is given by the non-negative integer solution $\vec{x}$ of the system \eqref{eq1.1}.
This is because for each $(n, \mathcal{L})$-type $\vec{\lambda}$ we could construct $x_{\vec{\lambda}}$ partitions formed by taking $|\vec{\lambda}|$ empty sets, and assigning a potential value $j\in \mathcal{L}$ to $\lambda_j$ of them.
Note that the total number of partitions is indeed
$$\sum_{\vec{\lambda}} x_{\vec{\lambda}}=\frac{(1, 2, \cdots, k) \cdot (A_\mathcal{L})^T \vec{x}}{n}=\frac{(1, 2, \cdots, k) \cdot \vec{b}_\mathcal{L}}{n} = M.$$
	
Now for the inductive step, assume that the statement holds for some $0\le \ell\le n-1$. We construct the following network.
	
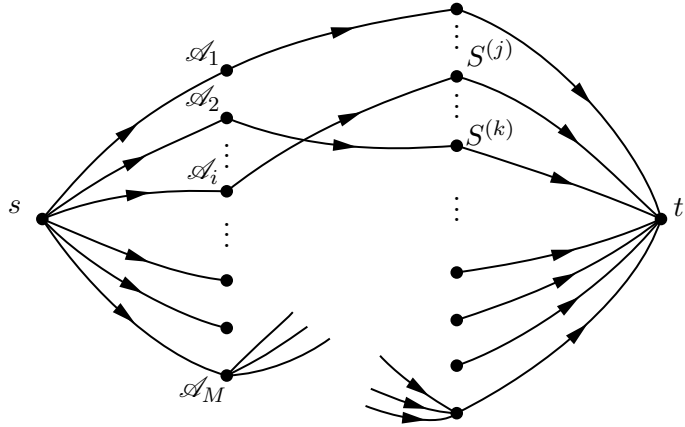
\begin{figure}[ht]
		\centering
		\tikzset{every picture/.style={line width=0.75pt}} %set default line width to 0.75pt

\begin{tikzpicture}[x=0.75pt,y=0.75pt,yscale=-1,xscale=1]
	%nodes of sets
			\filldraw [black] (343,28) circle (2pt);
			\filldraw [black] (343,62) circle (2pt);
			\filldraw [black] (343,97) circle (2pt);
			\filldraw [black] (343,161) circle (2pt);
			\filldraw [black] (343,185) circle (2pt);
			\filldraw [black] (343,208) circle (2pt);
			\filldraw [black] (343,232) circle (2pt);
			%nodes of partitions
			\filldraw [black] (227,59) circle (2pt);
			\filldraw [black] (227,83) circle (2pt);
			\filldraw [black] (227,120) circle (2pt);
			\filldraw [black] (227,165) circle (2pt);
			\filldraw [black] (227,189) circle (2pt);
			\filldraw [black] (227,213) circle (2pt);
			%nodes of source and sink
			\filldraw [black] (134,134) circle (2pt);
			\filldraw [black] (446,134) circle (2pt);
			%arcs from source
			\draw (134,134).. controls (163,99) and (182,81) .. (227,59);
			\draw (134,134).. controls (175,104.33) and (187,102) .. (227,83);
			\draw (134,134).. controls (175,117) and (214,120) .. (227,120);
			\draw (134,134).. controls (165,147) and (197,162) .. (227,165);
			\draw (134,134).. controls (157,159) and (203,184) .. (227,189);
			\draw (134,134).. controls (151,161) and (189,203) .. (227,213);
			
			\draw [shift={(182.17,85.91)}, rotate = 142.4] [fill={rgb, 255:red, 0; green, 0; blue, 0 }  ][line width=0.08]  [draw opacity=0] (12,-3) -- (0,0) -- (12,3) -- cycle    ;
			\draw [shift={(186.39,102.04)}, rotate = 153.82] [fill={rgb, 255:red, 0; green, 0; blue, 0 }  ][line width=0.08]  [draw opacity=0] (12,-3) -- (0,0) -- (12,3) -- cycle    ;
			\draw [shift={(188.28,120.5)}, rotate = 172] [fill={rgb, 255:red, 0; green, 0; blue, 0 }  ][line width=0.08]  [draw opacity=0] (12,-3) -- (0,0) -- (12,3) -- cycle    ;
			\draw [shift={(187.88,155.67)}, rotate = 199.37] [fill={rgb, 255:red, 0; green, 0; blue, 0 }  ][line width=0.08]  [draw opacity=0] (12,-3) -- (0,0) -- (12,3) -- cycle    ;
			\draw [shift={(184.76,171.44)}, rotate = 209.28] [fill={rgb, 255:red, 0; green, 0; blue, 0 }  ][line width=0.08]  [draw opacity=0] (12,-3) -- (0,0) -- (12,3) -- cycle    ;
			\draw [shift={(180.62,186.48)}, rotate = 220.35] [fill={rgb, 255:red, 0; green, 0; blue, 0 }  ][line width=0.08]  [draw opacity=0] (12,-3) -- (0,0) -- (12,3) -- cycle    ;

			%arcs to sink
			\draw (343,161).. controls (363,158) and (406,152).. (446,134);
			\draw (343,185).. controls (367,177) and (410,163) .. (446,134);
			\draw (343,208).. controls (370,199) and (417,170) .. (446,134);
			\draw (343,232).. controls (389,209) and (434,171).. (446,134);
			\draw (343,28).. controls (397,46) and (440,107)  .. (446,134);
			\draw (343,62).. controls (391,75) and (424,115) .. (446,134);
			\draw (343,97).. controls (382,109) and (422,123) .. (446,134);

			\draw [shift={(410.85,75.91)}, rotate = 227.62] [fill={rgb, 255:red, 0; green, 0; blue, 0 }  ][line width=0.08]  [draw opacity=0] (12,-3) -- (0,0) -- (12,3) -- cycle    ;
			\draw [shift={(404.8,96.2)}, rotate = 222] [fill={rgb, 255:red, 0; green, 0; blue, 0 }  ][line width=0.08]  [draw opacity=0] (12,-3) -- (0,0) -- (12,3) -- cycle    ;
			\draw [shift={(402.29,116.85)}, rotate = 199.81] [fill={rgb, 255:red, 0; green, 0; blue, 0 }  ][line width=0.08]  [draw opacity=0] (12,-3) -- (0,0) -- (12,3) -- cycle    ;
			\draw [shift={(402.79,149.46)}, rotate = 165.74] [fill={rgb, 255:red, 0; green, 0; blue, 0 }  ][line width=0.08]  [draw opacity=0] (12,-3) -- (0,0) -- (12,3) -- cycle    ;
			\draw [shift={(403.33,161.29)}, rotate = 154.11] [fill={rgb, 255:red, 0; green, 0; blue, 0 }  ][line width=0.08]  [draw opacity=0] (12,-3) -- (0,0) -- (12,3) -- cycle    ;
			\draw [shift={(404.96,173.25)}, rotate = 143.53] [fill={rgb, 255:red, 0; green, 0; blue, 0 }  ][line width=0.08]  [draw opacity=0] (12,-3) -- (0,0) -- (12,3) -- cycle    ;
			\draw [shift={(408.97,186.89)}, rotate = 137.48] [fill={rgb, 255:red, 0; green, 0; blue, 0 }  ][line width=0.08]  [draw opacity=0] (12,-3) -- (0,0) -- (12,3) -- cycle    ;
			
			\draw (227,59).. controls (266,39) and (323,31) .. (343,28);
			\draw (227,83).. controls (277,103) and (326,98) .. (343,97);
			\draw (227,120).. controls (265,90) and (303.67,75) .. (343,62);
			
			\draw [shift={(291.69,37.36)}, rotate = 166.73] [fill={rgb, 255:red, 0; green, 0; blue, 0 }  ][line width=0.08]  [draw opacity=0] (12,-3) -- (0,0) -- (12,3) -- cycle    ;
			\draw [shift={(295.19,80)}, rotate = 158] [fill={rgb, 255:red, 0; green, 0; blue, 0 }  ][line width=0.08]  [draw opacity=0] (12,-3) -- (0,0) -- (12,3) -- cycle    ;
			\draw [shift={(292.68,97.67)}, rotate = 184.95] [fill={rgb, 255:red, 0; green, 0; blue, 0 }  ][line width=0.08]  [draw opacity=0] (12,-3) -- (0,0) -- (12,3) -- cycle    ;
			
			\draw    (227,213) .. controls (250,201.67) and (255,197) .. (267.67,188.33) ;
			\draw    (227,213) .. controls (246.33,212.33) and (266.33,203.67) .. (279,194.33) ;
			\draw    (227,213) .. controls (243.67,195.67) and (255.67,186.33) .. (260.33,181) ;
			
			\draw    (299.67,219.67) .. controls (311,224.33) and (328.33,231.67) .. (343,232) ;
			\draw    (297,228.33) .. controls (309,233) and (333.67,239.67) .. (343,232) ;
			\draw    (304.33,203) .. controls (309.67,209) and (324.33,221.67) .. (343,232) ;
			\draw [shift={(328.22,229.83)}, rotate = 195.74] [fill={rgb, 255:red, 0; green, 0; blue, 0 }  ][line width=0.08]  [draw opacity=0] (12,-3) -- (0,0) -- (12,3) -- cycle    ;
			\draw [shift={(327.49,235.51)}, rotate = 185.93] [fill={rgb, 255:red, 0; green, 0; blue, 0 }  ][line width=0.08]  [draw opacity=0] (12,-3) -- (0,0) -- (12,3) -- cycle    ;
			\draw [shift={(329,223.26)}, rotate = 214.73] [fill={rgb, 255:red, 0; green, 0; blue, 0 }  ][line width=0.08]  [draw opacity=0] (12,-3) -- (0,0) -- (12,3) -- cycle    ;

			\node at(343,73)  [align=left] {$\displaystyle \vdots $};
			\node at(343,125) [align=left] {$\displaystyle \vdots $};
			\node at(343,38) [align=left] {$\displaystyle \vdots $};
			\node at(227,138) [align=left] {$\displaystyle \vdots $};
			\node at(227,98) [align=left] {$\displaystyle \vdots $};
			\node at(120,128)  {$s$};
			\node at(455,128) {$t$};
			\node at(215,50) {$\mathscr{A}_1$};
			\node at(215,73) {$\mathscr{A}_2$};
			\node at(215,110) {$\mathscr{A}_i$};
			\node at(215,220) {$\mathscr{A}_M$};
			
			\node at(360,50) {$S^{(j)}$};
			\node at(360,90) {$S^{(k)}$};

\end{tikzpicture}
	\end{figure}

Let $s$ be the source and $t$ be the sink. Each partition $\mathscr{A}_i$ defines a vertex and we add an arc from the source $s$ to each $\mathscr{A}_i$.
For each subset $S\subseteq[\ell]$, in the network we create vertices $S^{(j)}$ for all $j\in \mathcal{L}$ with $j\ge |S|$, where $S^{(j)}$ stands for the set $S$ with potential value $j$.
We add an arc from each $S^{(j)}$ to the sink $t$.
If $S$ occurs as a subset with potential value $j$ in the partition $\mathscr{A}_i$, then we add to the network an arc from $\mathscr{A}_i$ to $S^{(j)}$. Next we define the capacity function $\kappa$ as follows:
	$$\kappa(s,\mathscr{A}_i)=1,~~~~\kappa(\mathscr{A}_i,S^{(j)})=+\infty,~~\mbox{ and }~~\kappa(S^{(j)},t)=\binom{n-\ell-1}{j-1-|S|}.$$

Then we define a flow $f$ on this network as follows:
	$$f(s,\mathscr{A}_i)=1,~~~~f(\mathscr{A}_i,S^{(j)})=\frac{j-|S|}{n-\ell},~~\mbox{ and }~~f(S^{(j)},t)=\binom{n-\ell-1}{j-1-|S|}.$$
Let us check that $f$ is indeed a flow. It is easy to check that $0\le f(a) \le \kappa(a)$ for every arc $a$ in this network. To see that $f$ satisfies the conservation of flow, we consider the vertices $\mathscr{A}_i$ and $S^{(j)}$ separately:
	\begin{enumerate}
		\item For each $\vec{\lambda}$-partition $\mathscr{A}_i$ of $[\ell]$ with $\vec{\lambda}=(\lambda_1,\lambda_2,\dots,\lambda_k)$, the total value of flow leaving $\mathscr{A}_i$ is
		\begin{align*}
			\sum\limits_{S\in \mathscr{A}_i}\frac{j-|S|}{n-\ell}=\frac{\sum_{j\in \mathcal{L}}^{n}j\cdot\lambda_j-\sum\limits_{S\in \mathscr{A}_i}|S|}{n-\ell}=\frac{n-\ell}{n-\ell}=1=f(s,\mathscr{A}_i);
		\end{align*}
		\item For each $S^{(j)}$, by the inductive hypothesis, it appears in the partitions $\{\mathscr{A}_1, \cdots, \mathscr{A}_M\}$ for exactly $\binom{n-\ell}{j-|S|}$ times.
        So we have the total value of flow entering $S^{(j)}$ is $$\binom{n-\ell}{j-|S|}\cdot\frac{j-|S|}{n-\ell}=\binom{n-\ell-1}{j-1-|S|}=f(S^{(j)},t).$$
	\end{enumerate}
	Since $f(a)=\kappa(a)$ for all the edges $a$ leaving the source $s$, the so-defined $f$ must be a maximum flow. By Theorem \ref{DF}, there is an integral flow $f^*$ of the same maximum value. Therefore for each $\mathscr{A}_i$, we have $f^*(s,\mathscr{A}_i)=1$, and consequently, there is a unique arc $\mathscr{A}_i \to S^{(j)}$ with $f^*(\mathscr{A}_i,S^{(j)})=1$.
	As for each vertex $S^{(j)}$, we have $f^*(S^{(j)},t)=\binom{n-\ell-1}{j-1-|S|}$,	and by the conservation of flow, there are exactly $\binom{n-\ell-1}{j-1-|S|}$ arcs $a$ directed to $S^{(j)}$ with $f^*(a)=1$.
	Let us pay some attention to vertices $S^{(j)}$ with $j=|S|$ that
	as $f^*(S^{(j)},t)=\binom{n-\ell-1}{j-1-|S|}=0$,
	each arc $a$ directed to $S^{(j)}$ must have $f^*(a)=0$, hence it is impossible for any of the vertices $\mathscr{A}_i$ to have the unique arc $\mathscr{A}_i \to S^{(j)}$ with $f^*(\mathscr{A}_i,S^{(j)})=1$.

	Finally, we use $f^*$ to construct a desired collection $\mathscr{A}_1',\mathscr{A}_2',\dots, \mathscr{A}_M'$ of partitions of $[\ell+1]$. As mentioned above, every $\mathscr{A}_i$ has a unique $S^{(j)}$ such that $f^*(\mathscr{A}_i, S^{(j)}))=1$, where $j>|S|$ by the above discussion. Let $S'=S\cup \{\ell+1\}$, and update $\mathscr{A}_i$ by replacing $S$ with $S'$ and assigning to $S'$ the same potential value $j$.
	Note that we have $j\ge |S'|$.
	By definition, the new partition $\mathscr{A}_i'$ is still a $\vec{\lambda}$-partition of $[\ell+1]$.
	So the first and second properties for the new partitions $\mathscr{A}'_1, \cdots, \mathscr{A}'_M$ are satisfied.
	Since there are $\binom{n-\ell-1}{j-1-|S|}$ many arcs directed to $S^{(j)}$, the new set $S'=S\cup \{\ell+1\}$ with potential value $j$ (i.e. $S'^{(j)}$) is contained in exactly $\binom{n-\ell-1}{j-1-|S|}=\binom{n-(\ell+1)}{j-|S'|}$ new partitions $\mathscr{A}_i$'. For those $S \subset [\ell+1]$ not containing the element $\ell$, by induction they occur
	\begin{align*}
		&\binom{n-\ell}{j-|S|}-\binom{n-\ell-1}{j-1-|S|}=\binom{n-(\ell+1)}{j-|S|}
	\end{align*}
	times with potential value $j$ in the new partitions $\mathscr{A}'_1, \cdots, \mathscr{A}'_M$.
	This proves the third property for the new partitions $\mathscr{A}'_1, \cdots, \mathscr{A}'_M$.
	Hence the statement holds for $\ell+1$ and the proof is completed. \qed
%~\\~\\
%\noindent \textbf{Remark.} Actually a more general result is true by slightly modifying the this proof. If $\mathcal{F}$ consists of subsets of $[n]$ of size from a given set $\mathcal{L}$, then $\mathcal{F}$ is $1$-factorable if and only if $(A')^T\vec{x}=\vec{b'}$ has a non-negative solution. Here $A'$ has rows from types whose non-zero coordinates are from $\mathcal{L}$, and $b'_i=\binom{n}{i}$ for $i \in \mathcal{L}$ and $0$ otherwise.

\section{Finding non-negative integer solutions}\label{sec_solve}
After establishing Theorem \ref{thm_reduction}, to prove Theorem \ref{thm_main}, it remains to determine for which $n, k$, the system (for $\mathcal{L}=[k]$) of linear equations \eqref{eq1.1} has a non-negative integer solution.
For convenience, in this section the system \eqref{eq1.1} without specified $\mathcal{L}$ (or $A$ and $\vec{b}$ without subscripts) always means the case $\mathcal{L}=[k]$.
In the next two subsections we discuss the necessary and sufficient conditions respectively, and we prove Theorems \ref{thm_main} and \ref{thm_equiv} in the last subsection.

\subsection{Necessary condition}

The following lemma shows that it is necessary for $n$ to  come from certain congruence classes modulo $k$.

\begin{lemma}\label{th1.3}
	For any $n, k$ satisfying $2\le k< \frac{n}{2}$, if $n \not\equiv0,-1 \pmod k$, then the system $\eqref{eq1.1}$ does not have a non-negative real solution.
\end{lemma}

It turns out that unlike Baranyai's Theorem, the congruence condition in Lemma \ref{th1.3} alone is not enough to guarantee a $1$-factorization. For example, one could show that $K_{18}^{\le 6}$ is not $1$-factorable even though $18 \equiv  0\pmod 6$. More generally, if $n$ is not very large compared to $k$, even if it is from those congruence classes, it is still possible that the system \eqref{eq1.1}  does not have a non-negative integer solution. In these cases, there are not even non-negative real solutions.

\begin{lemma}\label{th1.4}
	For every positive integer $k \ge 2$, \\
(i) Suppose $n=j\cdot k+k$, where $j$ is a non-negative integer. If $2\le j \le k-4$, then the system $\eqref{eq1.1}$ has no non-negative solution.\\
(ii) Suppose $n=j\cdot k+k-1$, where $j$ is a non-negative integer. If $2\le j \le \lceil \frac{k}{2}\rceil -3$, then the system $\eqref{eq1.1}$ has no non-negative solution.
\end{lemma}

The proofs of Lemma \ref{th1.3} and \ref{th1.4} use Farkas' Lemma. In order to show that  $A^T \vec{x} = \vec{b}$ has no non-negative real solution, we will construct a hyperplane separating the convex cone formed by the column vectors of the matrix $A^T$ and the vector $\vec{b}$. Here we say a vector $\vec{x} \ge 0$ if each of its coordinates is non-negative.

\begin{lemma}[Farkas \cite{Farkas}]\label{lemma_farkas}
	Let $P \in \mathbb{R} ^{m\times n}$ and $\vec{b} \in \mathbb {R}^{n}$. Then exactly one of the following two assertions is true:
	\begin{enumerate}
		\item There exists a vector ${\displaystyle \vec {x} \in \mathbb {R} ^{m}}$ such that $P^T \vec{x} =\vec{b} $  and $\vec{x} \geq \vec{0}$.
		\item There exists a vector $\vec{y} \in \mathbb {R}^{n}$ such that $ P\vec{y} \geq \vec{0}$ and $\vec{b}^T\vec{y} <0$.
	\end{enumerate}
\end{lemma}
Below we give proofs to both Lemmas \ref{th1.3} and \ref{th1.4}.

\medskip
\noindent {\bf Proof of Lemma \ref{th1.3}}:
Let $n=j\cdot k+r$, $0<r<k-1$. We first consider the case when $j\ge3$. Let  %$$\vec{y}^{T}=(-1,\underbrace{\frac{j-1}{2},\dots,\frac{j-1}{2}}_{k-r-1},j,\underbrace{\frac{j}{2},\dots,\frac{j}{2}}_{r-1}).$$
$$\vec{y}^T = (\underbrace{\frac{j}{2}, \cdots, \frac{j}{2}}_{r-1}, j,   \underbrace{\frac{j-1}{2}, \cdots, \frac{j-1}{2}}_{k-r-1}, -1).$$
By Lemma \ref{lemma_farkas} we just need to prove $A\vec{y}\ge \vec{0}$ and $\vec{b}^T\vec{y} <0$.

Take an arbitrary row vector of $A$. By definition we know that it is of the form $\vec{\lambda}=(\lambda_1, \cdots, \lambda_k)$ such that $\sum_{i=1}^k i \cdot \lambda_i = n$ and all $\lambda_i$'s are non-negative integers. Since $n=jk+r$, we have $\lambda_{k}\le j$. Suppose $\lambda_k=j$, then $\sum_{i=1}^{k-1} i \cdot \lambda_i=n-kj=r>0$ and thus for some index $1 \le s \le r$, $\lambda_s$ must be strictly positive. Either we have at least two such $\lambda_s$ to be positive, or $\lambda_r=1$. In either case, $\sum_{i=1}^r \lambda_i y_i \ge j$. This already gives
$$\sum_{i=1}^{k} \lambda_i y_i \ge \left(\sum_{i=1}^r \lambda_i y_i\right) + \lambda_k y_k \ge j-j = 0.$$

Now we may assume that $\lambda_k \le j-1$. Similar as before, if at least two of $\lambda_1, \cdots \lambda_{r-1}$, or $\lambda_r$ itself is strictly positive, then $\sum_{i=1}^k \lambda_i y_i$ is non-negative. Otherwise suppose $\lambda_r=0$ and $\lambda_1+ \cdots + \lambda_{r-1} \le 1$, then
$$\sum_{i \neq r, 1 \le i \le k-1} i \cdot \lambda_i = n - k \lambda_k \ge jk+r -(j-1)k = k+r.$$
Since
$$\sum_{i \neq r, 1 \le i \le k-1} i \cdot \lambda_i \le (k-1)\sum_{i \neq r, 1 \le i \le k-1}\lambda_i.$$
We immediately have $\sum_{i \neq r, 1 \le i \le k-1} \lambda_i \ge 2$. Therefore we also have
$$\sum_{i=1}^k \lambda_i y_i \ge \frac{j-1}{2}\cdot \left(\sum_{i \neq r, 1 \le i \le k-1} \lambda_i\right) - \lambda_k \ge (j-1)-(j-1)=0.$$

Next we prove $\vec{y}^T\vec{b}<0$. Recall that $b_i=\binom{n}{i}$. So it suffices to verify that for $1 \le r \le k-2$ and $n=jk+r$ with $j \ge 3$,

\begin{equation}\label{ineq}
	\sum_{i=1}^{r-1} \frac{j}{2} \binom{n}{i}+ j \binom{n}{r}+\sum_{i=r+1}^{k-1} \frac{j-1}{2} \binom{n}{i}<\binom{n}{k}.
\end{equation}
\begin{comment}
It is easy to check that in this range of $n$, $\binom{n}{i-1} \le \frac{1}{2} \binom{n}{i}$ for $1 \le i \le k$. Therefore it suffices to prove
\begin{equation}\label{ineq}
	2j \binom{n}{r}+\sum_{i=r+1}^{k-1} \frac{j-1}{2} \binom{n}{i}<\binom{n}{k}.
\end{equation}

We fix $j, k$ and prove inequality \eqref{ineq} using a reverse induction on $r$ (which would also change the value of $n=jk+r$), namely the base case is when $r=k-2$. In this case, by dividing $\binom{n}{k}$, \eqref{ineq} simplifies to
$$2j \cdot \frac{k(k-1)}{(n-k+1)(n-k+2)} + \frac{j-1}{2} \cdot \frac{k}{n-k+1}<1.$$
Plugging $n=jk+k-2$, the left hand side becomes
$$2j \cdot \frac{k(k-1)}{(jk-1)(jk)} + \frac{j-1}{2} \cdot \frac{k}{jk-1}=\frac{\frac{j+3}{2}k-2}{jk-1}<1,$$
which proves the base case $r=k-2$. Now suppose inequality \eqref{ineq} also for $r \ge 2$, using the Pascal identity $\binom{m}{t}=\binom{m-1}{t}+\binom{m-1}{t-1}$, we could rewrite \eqref{ineq} as
$$2j\left(\binom{n-1}{r}+\binom{n-1}{r-1}\right) + \sum_{i=r+1}^{k-1} \frac{j-1}{2} \left(\binom{n-1}{i}+\binom{n-1}{i-1}\right)<\binom{n-1}{k}+\binom{n-1}{k-1}.$$
This implies
\begin{align*}
\binom{n-1}{k}&>\frac{j-3}{2}\binom{n-1}{k-1} + (j-1) \sum_{i=r+1}^{k-2} \binom{n-1}{i} + \frac{5j-1}{2} \binom{n-1}{r}+ 2j \binom{n-1}{r-1}
\end{align*}
\end{comment}

%\begin{comment}

Note that $$\binom{n}{k-i+1}=\frac{j-1}{2}\binom{n}{k-i}+\frac{j-1}{2}\binom{n}{k-i}+\frac{r+1+j (i-1)}{k-(i-1)}\binom{n}{k-i}.$$
By substituting one of the $\binom{n}{k-i}$ by the above identity for $i+1$ and repeat this process, we obtain
\begin{equation}\label{eq2.1}
	\binom{n}{k}=\sum_{i=1}^{k-1}\left(\frac{j-1}{2}\right)^i\binom{n}{k-i}+\left(\frac{j-1}{2}\right)^{k-1}\binom{n}{1}+\sum_{i=1}^{k-1}\left(\frac{j-1}{2}\right)^{i-1}\frac{r+1+j (i-1)}{k-(i-1)}\binom{n}{k-i}.
\end{equation}
We compare the coefficients of $\binom{n}{k-i}$ in this expression with the left hand side of \eqref{ineq}. Note that $1\le r \le k-2$, we discuss the following cases according to the value of $j$.
\begin{enumerate}
	\item For $j \ge 6$, using that for $i \ge k-r \ge 2$, $(\frac{j-1}{2})^i \ge (\frac{j-1}{2})^2>j$,
	\begin{align*}
		\binom{n}{k}>\sum_{i=1}^{k-1}\left(\frac{j-1}{2}\right)^i\binom{n}{k-i}>\sum_{i=1}^{k-r-1}\frac{j-1}{2}\binom{n}{k-i}+\sum_{i=k-r}^{k-1}j\binom{n}{k-i}.
	\end{align*}
	\item For $j=5$, note that $(\frac{j-1}{2})^3-j>0$, we could establish the same inequality when $r\le k-3$. It suffices to check the case $r=k-2$ and compare the coefficients of $\binom{n}{k-2}$. Here we also involve the last sum on the right hand side of \eqref{eq2.1}. Note that in \eqref{eq2.1} the coefficients of $\binom{n}{k-2}$ add up to be $2^2+2\cdot\frac{k-2+1+5(2-1)}{k-(2-1)}>5=j$. It completes the proof of the $j=5$ case.
	
	\item For $j=4$, since $(\frac{j-1}{2})^4-j>0$, like in Case 2, it suffices to check the cases $r=k-2$ and $r=k-3$. For $r=k-3$, the coefficient of $\binom{n}{k-3}$ from \eqref{eq2.1} is $(\frac{j-1}{2})^3+(\frac{j-1}{2})^2 \cdot \frac{(k-3)+1+j (3-1)}{k-(3-1)}$, which is greater than $j$. For $r=k-2$, the coefficient of $\binom{n}{k-3}$ from \eqref{eq2.1} is $(\frac{j-1}{2})^3+(\frac{j-1}{2})^2 \cdot\frac{(k-2)+1+j (3-1)}{k-(3-1)}$, greater than $j$. The coefficient of $\binom{n}{k-2}$ from \eqref{eq2.1} is $(\frac{j-1}{2})^2+(\frac{j-1}{2}) \cdot\frac{(k-2)+1+j (2-1)}{k-(2-1)}$, combined with the surplus term $\frac{(k-2)+1}{k} \cdot \binom{n}{k-1} \ge \frac{(k-2)+1}{k} \cdot  \binom{n}{k-2}$, which is not hard to check that $\frac{9}{4}+\frac{3}{2} \cdot \frac{k+3}{k-1} + \frac{k-1}{k}$ is greater than $j=4$ for all $k$.

	\item When $j=3$, we compare the identity \eqref{eq2.1} with inequality \eqref{ineq}, note that what we need to prove is
	\begin{equation}\label{ineq_new}
		\binom{n}{1}+\sum_{i=1}^{k-1}\frac{r+1+3 (i-1)}{k-(i-1)}\binom{n}{k-i}>2\binom{n}{r}+\sum_{i=k-r+1}^{k-1}\frac{1}{2}\binom{n}{k-i}.
	\end{equation}
	Note that $\binom{n}{k-i}/\binom{n}{k-i-1}=(n-k+i+1)/(k-i)>2$, therefore the right hand side is at most $2 \binom{n}{r} + \binom{n}{r-1}$, while the left hand side is at least
	$$\frac{r+1}{k}\binom{n}{k-1} + \frac{r+4}{k-1}\binom{n}{k-2}+\frac{r+7}{k-2}\binom{n}{k-3}.$$
	Note that
	$$\frac{\binom{n}{k-1}}{\binom{n}{r}}=\frac{(n-k+2) \cdots (n-r)}{(r+1) \cdots (k-1)} \ge  \frac{n-k+2}{r+1} \ge \frac{k}{r+1}.$$
	For $r=k-2$, $\frac{r+4}{k-1}\binom{n}{k-2}=\frac{k+2}{k-1}\binom{n}{k-2}>\binom{n}{k-2}$, and for $r \le k-3$,
	$$\frac{\binom{n}{k-2}}{\binom{n}{r}}=\frac{(n-k+3) \cdots (n-r)}{(r+1) \cdots (k-2)}   \ge \frac{n-k+3}{r+1} \ge \frac{k-1}{r+4}.$$
	Finally, when $r=k-2$, $\frac{r+7}{k-2}\binom{n}{k-3}=\frac{k+5}{k-2}\binom{n}{k-3}>\binom{n}{k-3}$, and for $r \le k-3$,
	$$\frac{\binom{n}{k-3}}{\binom{n}{r-1}}=\frac{(n-k+4) \cdots (n-r+1)}{r \cdots (k-3)}   \ge \frac{n-k+4}{r} \ge \frac{k-2}{r+7}.$$
	These three inequalities immediately imply inequality \eqref{ineq_new}.

	%For $r=k-2$, the left hand side contains the sum $\frac{k-1}{k}\binom{n}{k-1} + \frac{k+2}{k-1}\binom{n}{k-2}+\frac{k+5}{k-2}\binom{n}{k-3}$, which is strictly greater than $2\binom{n}{k-2} + \binom{n}{k-3}$.

\end{enumerate}
%\end{comment}

For $j=2$, we will use a different hyperplane when applying Farkas' Lemma. Now we have $n=2k+r$, $1 \le r \le k-2$. We define $\vec{y}$ as follows. We always take $y_1 = \cdots = y_{r-1}=1$, $y_r=2$, and $y_k=-1$. For $r \not \equiv k \pmod 2$, we take $y_{r+1} = \cdots = y_{\lfloor(r+k)/2\rfloor}=1$, and for $r \equiv k \pmod 2$, we take we take $y_{r+1} = \cdots = y_{\lfloor(r+k)/2\rfloor-1}=1$, $y_{\lfloor(r+k)/2\rfloor}=1/2$. The remaining $y_i$'s are set to be $0$.

First we prove $\sum_{i=1}^k \lambda_i y_i \ge 0$. This is obviously true if $\lambda_k=0$. Suppose $\lambda_k=2$, then $\sum_{i=1}^{k-1} i \lambda_i = r$. In this case we either have $\sum_{i=1}^{r-1} \lambda_i \ge 2$ or $\lambda_r=1$, both implying $\sum_{i=1}^k \lambda_i y_i \ge 0$. The remaining case is $\lambda_k=1$ and we have $\sum_{i=1}^{k-1} i \lambda_i = k+r$.
If there exists $i \le r$ with $\lambda_i \ge 1$ then we are already done. Now suppose $\lambda_1=\cdots= \lambda_r=0$, we have $\sum_{i=r+1}^{k-1} i \lambda_i =k+r$. Either (in the case when $k$ and $r$ have the same parity) we have $\lambda_{(k+r)/2}=2$, or there exists some $i<(k+r)/2$ such that $\lambda_i=1$. By our choice of $\vec{y}$, it is not hard to see that once again $\sum_{i=1}^k \lambda_i y_i \ge 0$.

Finally let us prove that $\vec{y}^T\vec{b}<0$. For $r \equiv k \pmod 2$, $1 \le r \le k-2$ and $n=2k+r$, we need to show
\begin{equation}\label{ineq_even}
\left(\sum_{i=1}^{(k+r)/2-1} \binom{n}{i}\right) + \frac{1}{2} \binom{n}{(k+r)/2} + \binom{n}{r} < \binom{n}{k}
\end{equation}
For $i \le (k+r)/2$ we have
$$\binom{n}{i-1}/\binom{n}{i}=\frac{i}{n-i+1} \le \frac{(k+r)/2}{(3k+r)/2+1}\le \frac{k-1}{2k}.$$
Also $\binom{n}{r} \le \binom{n}{(k+r)/2-1}$, thus the left hand side of \eqref{ineq_even} is at most
\begin{align*}
&~~\left(\frac{k-1}{2k}+\left(\frac{k-1}{2k}\right)^2+\cdots\right) \binom{n}{(k+r)/2} + \frac{1}{2}\binom{n}{(k+r)/2}+\frac{k-1}{2k}\binom{n}{(k+r)/2}\\
&<\frac{k-1}{k+1} \binom{n}{(k+r)/2}+ \frac{2k-1}{2k}\binom{n}{(k+r)/2} =\frac{4k^2-k-1}{2k^2+2k}\binom{n}{(k+r)/2} \le \binom{n}{k}.
\end{align*}
The last inequality follows from
\begin{align*}
\binom{n}{k}/\binom{n}{(k+r)/2}&=\frac{(n-k+1)\cdots (n-(k+r)/2)}{(\frac{k+r}{2}+1) \cdots k} \ge \frac{n-(k+r)/2}{(k+r)/2 +1}\\
&=\frac{3k+r}{k+r+2} \ge \frac{3k+(k-2)}{k+(k-2)+2}=\frac{2k-1}{k}.
\end{align*}
For the case $r \not\equiv k \pmod 2$, in this case we have $r \le k-3$, we need to show
\begin{equation}\label{ineq_odd}
	  \binom{n}{r}+\sum_{i=1}^{(k+r-1)/2} \binom{n}{i} < \binom{n}{k}
\end{equation}
For $i \le (k+r+1)/2$,
$$\binom{n}{i-1}/\binom{n}{i}=\frac{i}{n-i+1} \le \frac{k+r+1}{3k+r+1}\le \frac{k+(k-3)+1}{3k+(k-3)+1}=\frac{k-1}{2k-1}.$$
Also $\binom{n}{r} \le \binom{n}{(k+r-1)/2} \le \frac{k-1}{2k-1}\binom{n}{(k+r+1)/2}$. So the left hand side of \eqref{ineq_odd} is at most
\begin{align*}
&~~\left(\frac{k-1}{2k-1}+\left(\frac{k-1}{2k-1}\right)^2 + \cdots \right) \binom{n}{(k+r+1)/2} + \binom{n}{r}\\
&< \frac{k-1}{k} \binom{n}{(k+r+1)/2} + \frac{k-1}{2k-1}\binom{n}{(k+r+1)/2} < \binom{n}{k}.
\end{align*}
The last inequality follows from
\begin{align*}
	\binom{n}{k}/\binom{n}{(k+r+1)/2}&=\frac{(n-k+1)\cdots (n-(k+r+1)/2)}{(k+r+3)/2 \cdots k} \ge \frac{n-(k+r+1)/2}{(k+r+3)/2}\\
	&=\frac{3k+r-1}{k+r+3} \ge \frac{3k+(k-3)-1}{k+(k-3)+3}=\frac{2k-2}{k}>\frac{k-1}{k}+\frac{k-1}{2k-1}.
\end{align*}
\qed

\medskip

\noindent {\bf Proof of Lemma \ref{th1.4}}:
	(i) We first consider the case $n=j k+k$, $1\le j\le k-4$. Let  $$\vec{y}^T={(\underbrace{j+1,\dots,j+1}_{j+1},\underbrace{j/2,\dots,j/2}_{k-j-3},-1,0)}.$$
	Then once again from Farkas' Lemma, we just need to verify $A\vec{y}\ge \vec{0}$ and $\vec{b}^T\vec{y} <0$.
	
	For each row vector $\vec{\lambda}=(\lambda_1, \cdots, \lambda_k)$ of $A$, we know $\sum_{i=1}^{k}i \lambda_{i}=n$. Since $n=jk+k={\left(j+1\right)}{\left(k-1\right)}+j+1$ and $1\le j\le k-4$, we know that $\lambda_{k}+\lambda_{k-1}\le j+1$ and $\lambda_{k}+\lambda_{k-1}\le j$ if $\sum_{i=j+2}^{k-2}\lambda_{i}>0$. We discuss these two cases separately:
	\begin{enumerate}
	\item  When $\sum_{i=j+2}^{k-2}\lambda_{i}=0$. If we also have $\sum_{i=1}^{j+1}\lambda_{i}=0$, then  $\vec{\lambda}={(0,\dots,0, j+1)}$ and trivially $\vec{\lambda}\vec{y}\ge 0$. Now suppose $\sum_{i=1}^{j+1}\lambda_{i}\ge 1$, this case is also trivial since $\lambda_{k}+\lambda_{k-1}\le j+1$.

	\item When  $\sum_{i=j+2}^{k-2}\lambda_{i}>0$. If $\sum_{i=j+2}^{k-2}\lambda_{i}\ge 2$, observe that $\lambda_{k}+\lambda_{k-1}\le j$ and we also have $\vec{\lambda}\vec{y}\ge 2 \cdot (j/2) - j \ge 0$. Thus we just need to consider the case $\sum_{i=j+2}^{k-2}\lambda_{i}=1$. If $\sum_{i=1}^{j+1}\lambda_{i}=0$ then $(k-1) \lambda_{k-1}+k\lambda_k  \in \{(jk+2, (j+1)k-(j+2)\}$, which implies that $\lambda_{k-1} \in \{j+2, \cdots, k-2\}$, contradicting $\lambda_k+\lambda_{k-1} \le j$. Therefore $\sum_{i=1}^{j+1}\lambda_{i}\ge 1$, and this implies $\vec{\lambda}\vec{y}\ge (j+1)-j >0$.
	\end{enumerate}

   To apply Farkas' Lemma, we will also need to show that $\vec{b}^T\vec{y}<0$. We know $b_i=\binom{n}{i}$ in the system $\eqref{eq1.1}$. First of all, similar as before, $$\binom{n}{k-i+1}=\frac{j}{2}\binom{n}{k-i}+\frac{j}{2}\binom{n}{k-i}+\frac{1+{\left(j+1\right)} {\left(i-1\right)}}{k-{\left(i-1\right)}}\binom{n}{k-i}.$$Thus by repeatedly applying this identity, we have
	\begin{equation}\label{eq2.3}
		\binom{n}{k-1}=\sum_{i=2}^{k-1}{\left(\frac{j}{2}\right)}^{i-1}\binom{n}{k-i}+{\left(\frac{j}{2}\right)}^{k-2}\binom{n}{1}+\sum_{i=2}^{k-1}{\left(\frac{j}{2}\right)}^{i-2}\frac{1+{\left(j+1\right)} {\left(i-1\right)}}{k-{\left(i-1\right)}}\binom{n}{k-i}.
	\end{equation}
	What we need to prove is
	\begin{equation}\label{ineq2.4}
		\binom{n}{k-1}>\sum_{i=2}^{k-1}\frac{j}{2}\binom{n}{k-i}+\sum_{i=k-j-1}^{k-1}\frac{j+2}{2}\binom{n}{k-i}.
	\end{equation}
	Notice that $j+1 \le k-3$ since $2\le j \le k-4$, we divide our discussion into several cases based on the value of $j$.
	\begin{enumerate}
		\item For $j\ge 5$, we have ${\left(\frac{j}{2}\right)}^2-j-1>0$ and thus
		\begin{align*}
			\binom{n}{k-1}&=\sum_{i=2}^{k-1}{\left(\frac{j}{2}\right)}^{i-1}\binom{n}{k-i}+{\left(\frac{j}{2}\right)}^{k-2}\binom{n}{1}+\sum_{i=2}^{k-1}{\left(\frac{j}{2}\right)}^{i-2}\frac{1+{\left(j+1\right)} {\left(i-1\right)}}{k-{\left(i-1\right)}}\binom{n}{k-i}\\
			&>\sum_{i=2}^{k-1}{\left(\frac{j}{2}\right)}^{i-1}\binom{n}{k-i}>\sum_{i=2}^{k-1}\frac{j}{2}\binom{n}{k-i}+\sum_{i=k-j-1}^{k-1}\frac{j+2}{2}\binom{n}{k-i}.
		\end{align*}
		\item For $j=4$, we have ${\left(\frac{j}{2}\right)}^3-j-1>0$. It is easy to verify the inequality \eqref{ineq2.4} for $j+1\le k-4$. So it remains to check the case when $j+1=k-3$ (i.e. $k=j+4=8$ and $n=jk+k=40$). We have $\vec{y}^T={\left(5,5,5,5,5,2,-1,0\right)}$ and $\vec{b}^T={\left(\binom{40}{8},\dots,\binom{40}{1}\right)}$ when $k=j+4=8$ and $n=jk+k=40$. We have $\vec{b}^T\vec{y}<0$ by calculation.
		\item For $j=3$, note that ${\left(\frac{j}{2}\right)}^4-j-1>0$, it is easy to check that the inequality \eqref{ineq2.4} when $j+1\le k-5$. So we just need to check the case that $j+1=k-3$ and $j+1=k-4$. We have $\vec{y}^T={\left(4,4,4,4,1.5,-1,0\right)}$ and $\vec{b}^T={\left(\binom{28}{7},\dots,\binom{28}{1}\right)}$ when $k=j+4=7$ and $n=jk+k=28$. We have $\vec{b}^T\vec{y}<0$ by calculation. Similarly, we have $\vec{y}^T={\left(4,4,4,4,1.5,1.5,-1,0\right)}$ and $\vec{b}^T={\left(\binom{32}{8},\dots,\binom{32}{1}\right)}$ when $k=j+5=8$ and $n=jk+k=32$. Both cases have $\vec{b}^T\vec{y}<0$ by calculation.
		\item For $j=2$, we compare the identity \eqref{eq2.3} with the inequality \eqref{ineq2.4}, after canceling some terms, what we need to prove is
		\begin{equation}\label{ineq2.5}
			\binom{n}{1}+\sum_{i=2}^{k-1}\frac{1+3 {\left(i-1\right)}}{k-{\left(i-1\right)}}\binom{n}{k-i}>2\left(\binom{n}{3}+\binom{n}{2}+\binom{n}{1}\right).
		\end{equation}
		Note that $\frac{1+3(i-1)}{k-(i-1)}$ is increasing in $i$. Furthermore setting $i=k-3$, we have $\frac{1+3 {\left(i-1\right)}}{k-{\left(i-1\right)}}=\frac{1+3 {\left(k-3-1\right)}}{k-{\left(k-3-1\right)}}=\frac{3k-11}{4}$. If $\frac{3k-11}{4}\ge 2$, i.e. $k\ge7$, we have proved the inequality \eqref{ineq2.5}. Since $j\le k-4$, we also have $k\ge6$. Thus we just need to check the case that $j=2$ and $k=6$. We have $\vec{y}^T={\left(3,3,3,1,-1,0\right)}$ and $\vec{b}^T={\left(\binom{18}{6},\dots,\binom{18}{1}\right)}$. It is easy to check that $\vec{b}^T\vec{y}<0$ by calculation.
	\end{enumerate}
	\bigskip
	
	\noindent (ii) Next we consider the case $n=j\cdot k+k-1$, $1\le j \le \lceil \frac{k}{2}\rceil -3$. Let $$\vec{y}^T={(\underbrace{j+1,\dots,j+1}_{2j+1},\underbrace{j/2,\dots,j/2}_{k-2j-4},-1,j,-1 )}.$$ Then we just need to prove $A\vec{y}\ge \vec{0}$ and $\vec{b}^T\vec{y} <0$. We start by checking $A\vec{y} \ge \vec{0}$.
	
	For each row vector $\vec{\lambda}=(\lambda_1, \cdots, \lambda_k)$ of $A$, by definition we have $\sum_{i=1}^{k} i \lambda_{i}=n$. Since $n=j\cdot k+k-1={\left(j+1\right)}\cdot {\left(k-2\right)}+2j+1$ and $1\le j \le \lceil \frac{k}{2}\rceil -3$, we have $\lambda_{k}+\lambda_{k-2}\le j+1$ and $\lambda_{k}+\lambda_{k-2}\le j$ if $(\sum_{i=2j+2}^{k-3}\lambda_{i})+\lambda_{k-1}>0$.  Then we consider the following two cases:
	\begin{enumerate}
	\item When $\sum_{i=2j+2}^{k-3}\lambda_{i}=0$. If $\sum_{i=1}^{2j+1}\lambda_{i}=0$, note that the equation $(k-2)\lambda_{k-2}+k\lambda_k=jk+k-1$ has no non-negative integer solution. Therefore the only possibility is $\vec{\lambda}={(\underbrace{0,\dots,0}_{k-2},1,j)}$, which gives $\vec{\lambda}\vec{y}\ge 0$ trivially. The case $\sum_{i=1}^{2j+1}\lambda_{i}\ge 1$ is also trivial since $\lambda_{k}+\lambda_{k-2}\le j+1$.
	\item	Now suppose $\sum_{i=2j+2}^{k-3}\lambda_{i}>0$. If $\sum_{i=2j+2}^{k-3}\lambda_{i}\ge 2$, then $\lambda_{k}+\lambda_{k-2}\le j$, and thus  $\vec{\lambda}\vec{y}\ge 0$. So we just need to consider the case $\sum_{i=2j+2}^{k-3}\lambda_{i}=1$. To prove $\vec{\lambda}\vec{y}\ge 0$, we just need to show $(\sum_{i=1}^{2j+1}\lambda_{i})+\lambda_{k-1} \ge  1$. This is true since the inequality $n-(k-3) \le (k-2)\lambda_{k-2}+k \lambda_k \le n-(2j+2)$ has no non-negative integers solution for $1 \le j \le \lceil k/2 \rceil -3$.
	\end{enumerate}

	Next we will prove that $\vec{b}^T\vec{y}<0$. Recall that $b_i=\binom{n}{i}$ in the system $\eqref{eq1.1}$. It is easy to check that $$\binom{n}{k-i+1}=\frac{j}{2}\binom{n}{k-i}+\frac{j}{2}\binom{n}{k-i}+\frac{{\left(j+1\right)} {\left(i-1\right)}}{k-{\left(i-1\right)}}\binom{n}{k-i}.$$Thus we have
	\begin{equation}\label{eq2.6}
		\binom{n}{k-2}=\sum_{i=3}^{k-1}{\left(\frac{j}{2}\right)}^{i-2}\binom{n}{k-i}+{\left(\frac{j}{2}\right)}^{k-3}\binom{n}{1}+\sum_{i=3}^{k-1}{\left(\frac{j}{2}\right)}^{i-3}\frac{{\left(j+1\right)} {\left(i-1\right)}}{k-{\left(i-1\right)}}\binom{n}{k-i}.
	\end{equation}
	Observe that  $\binom{n}{k}=j \binom{n}{k-1}$ when $n=jk+(k-1)$, so what we need to prove is
	\begin{equation}\label{ineq2.7}
		\binom{n}{k-2}>\sum_{i=3}^{k-1}\frac{j}{2}\binom{n}{k-i}+\sum_{i=k-2j-1}^{k-1}\frac{j+2}{2}\binom{n}{k-i}.
	\end{equation}
	Note that $2j+1 \le k-4$ since $2\le j \le \lceil \frac{k}{2}\rceil -3$, we discuss the following cases according to the value of $j$.
	\begin{enumerate}
		\item For $j \ge 5$, we have ${\left(\frac{j}{2}\right)}^2-j-1>0$, therefore
		\begin{align*}
			\binom{n}{k-2}&=\sum_{i=3}^{k-1}{\left(\frac{j}{2}\right)}^{i-2}\binom{n}{k-i}+{\left(\frac{j}{2}\right)}^{k-3}\binom{n}{1}+\sum_{i=3}^{k-1}{\left(\frac{j}{2}\right)}^{i-3}\frac{{\left(j+1\right)} {\left(i-1\right)}}{k-{\left(i-1\right)}}\binom{n}{k-i}\\
			&>\sum_{i=3}^{k-1}{\left(\frac{j}{2}\right)}^{i-2}\binom{n}{k-i}>\sum_{i=3}^{k-1}\frac{j}{2}\binom{n}{k-i}+\sum_{i=k-2j-1}^{k-1}\frac{j+2}{2}\binom{n}{k-i}.
		\end{align*}
		\item For $j=4$, we have ${\left(\frac{j}{2}\right)}^3-j-1>0$, it is easy to check that the inequality \eqref{ineq2.7} similarly with Case $1$ when $2j+1\le k-5$.  So we just need to check the case that $2j+1=k-4$, which gives $k=2j+5=13$ and $n=jk+k-1=64$. We have $\vec{y}^T={(5,5,5,5,5,5,5,5,5,2,-1,4,-1)}$ and $\vec{b}^T={\left(\binom{64}{13},\dots,\binom{64}{1}\right)}$. By calculation $\vec{b}^T\vec{y}<0$.
		\item For $j=3$, we have ${\left(\frac{j}{2}\right)}^4-j-1>0$. it is easy to check that the inequality \eqref{ineq2.7} similarly with case $1$ when $2j+1\le k-6$. So we just need to check the case that $2j+1=k-4$ and $2j+1=k-5$. We have $\vec{y}^T={(4,4,4,4,4,4,4,1.5,-1,3,-1)}$ and $\vec{b}^T={\left(\binom{43}{11},\dots,\binom{43}{1}\right)}$ when $k=2j+5=11$ and $n=jk+k-1=43$. Similarly, we have $\vec{y}^T={(4,4,4,4,4,4,4,1.5,1.5,-1,3,-1)}$ and $\vec{b}^T={\left(\binom{47}{12},\dots,\binom{47}{1}\right)}$ when $k=2j+6=12$ and $n=jk+k-1=47$. In both cases, calculations give $\vec{b}^T\vec{y}<0$.
		\item For $j=2$, we compare the identity \eqref{eq2.6} with the inequality \eqref{ineq2.7} and note that it suffices to prove
		\begin{equation}\label{ineq2.8}
			\binom{n}{1}+\sum_{i=3}^{k-1}\frac{3 {\left(i-1\right)}}{k-{\left(i-1\right)}}\binom{n}{k-i}>2\sum_{i=1}^{5}\binom{n}{i}.
		\end{equation}
		Note that $\frac{3 {\left(i-1\right)}}{k-{\left(i-1\right)}}$ is monotone increasing in $i$. So we could focus our attention on the coefficient of the term $\binom{n}{5}$. Note that when $i=k-5$, $\frac{3 {\left(i-1\right)}}{k-{\left(i-1\right)}}=\frac{3k-18}{6}$. If $\frac{3k-18}{6}\ge 2$, i.e. $k \ge 10$, the inequality \eqref{ineq2.8} is obviously true. Since $j\le\lceil \frac{k}{2}\rceil -3$, we also have $k\ge9$. Thus it remains to check the case that $j=2$ and $k=9$. We have $\vec{y}^T={\left(3,3,3,3,3,1,-1,2,-1\right)}$ and $\vec{b}^T={\left(\binom{26}{9},\dots,\binom{26}{1}\right)}$. It is easy to verify $\vec{b}^T\vec{y}<0$ by calculation. \qed
	\end{enumerate}

\subsection{Sufficient condition}
In the previous subsection, we have found necessary conditions for the system \eqref{eq1.1} to have a non-negative integer solution. The following two lemmas below show that these necessary conditions are indeed sufficient.

\begin{lemma}\label{lem_solve1}
Suppose $n > 2k$ and $n=j\cdot k+k$, where $j$ is a non-negative integer such that $j \ge k-3$, then the system \eqref{eq1.1} has a non-negative integer solution.
\end{lemma}

\begin{proof}
Suppose $n=jk+k$ and $j \ge k-3$. So $n$ is at least $k^2-2k$ and divisible by $k$.
Below we construct an explicit non-negative integer solution to the system \eqref{eq1.1}.
First we consider the case when $n \ge k^2-k$.
For each $i\in \{1, \cdots, k-1\}$, we consider $\vec{\lambda}_{\texttt{i}}=(\lambda_1,\lambda_2,\cdots,\lambda_k) \in \mathbb{Z}_{\ge 0}^k$ such that
$$\lambda_i=\frac{k}{\gcd(k, i)},~~ \lambda_k=\frac{n}{k}-\frac{i}{\gcd(k, i)},~\mbox{ and }~\lambda_s=0~\textrm{for~}s \not\in \{i, k\}.$$
It is easy to check that $\sum_{\ell=1}^k \ell \lambda_{\ell}=n$.
Also note that $$\lambda_k = n/k -i/\gcd(k, i) \ge n/k-i \ge n/k-(k-1) \ge 0.$$
So $\vec{\lambda}_{\texttt{i}}$ is an $(n, k)$-type and thus appears as a row of the matrix $A$.
We will assign $\binom{n}{i}/(k/\gcd(k, i))$ to $x_{\vec{\lambda}_{\texttt{i}}}$, but first we explain why $\binom{n}{i}/(k/\gcd(k, i))$ is an integer. For every prime $p$,  if the $p$-adic valuation of $n$ is $x$ and that of $i$ is $y$, then the $p$-adic valuation of $n/\gcd(n, i)$ equals $\max\{x-y, 0\}$.
On the other hand, Kummer's Theorem \cite{Kummer} says that for any prime $p$, the $p$-adic valuation of $\binom{n}{i}$ is equal to the number of carries when $n-i$ is added to $i$ in base $p$. This quantity is obviously at least $\max \{x-y, 0\}$. Consequently $n/\gcd(n, i)$ always divides $\binom{n}{i}$. Since $n$ is a multiple of $k$, it is easy to see that $k/\gcd(k, i)$ divides $n/\gcd(n, i)$. This completes the proof of our claim.

Once we have chosen the value of $x_{\vec{\lambda}_{\texttt{i}}}$ for the aforementioned $\vec{\lambda}_{\texttt{i}}$ for $i\in \{1,\cdots,k-1\}$, the first $k-1$ equations of the system \eqref{eq1.1} have already been taken care of.
We just use the type $\vec{\lambda}_{\texttt{k}}=(\lambda_1,\cdots,\lambda_k) $ with $\lambda_k=n/k$ and $\lambda_s=0$ for $s \neq k$ to cover whatever is left for $b_k=\binom{n}{k}$.
Recall that $\binom{n}{k-i-1}/\binom{n}{k-i} \le k/(n-k+1)$ for $i\ge 0$, and $n \ge k^2-k$, therefore
the accumulated value $\gamma_{n,k}$ of the last equation of the system \eqref{eq1.1} equals
\begin{align*}
\sum_{i=1}^{k-1} \frac{n/k-i/\gcd(k, i)}{k/\gcd(k, i)}\cdot\binom{n}{i}
&=\sum_{i=1}^{k-1} \left(\frac{n}{k^2/\gcd(k, i)}-\frac{i}{k}\right)\cdot\binom{n}{i}
 \le \sum_{i=1}^{k-1} \left(\frac{n}{2k}-\frac{1}{k}\right)\cdot\binom{n}{i}\\
 &\le \left(\frac{n}{2k}-\frac{1}{k}\right)\left(\frac{k}{n-k+1}+\left(\frac{k}{n-k+1}\right)^2+\cdots \right)\binom{n}{k}\\
 &< \frac{n-2}{2(n-2k+1)} \binom{n}{k} \le \frac{k^2-k-2}{2(k^2-3k+1)} \binom{n}{k} \le \binom{n}{k}=b_k
\end{align*}
if $k \ge 4$. For $k=2$ the left hand side is $(n/2-1)n/2$, which is always less than $\binom{n}{2}$.
For $k=3$, the left hand side is $(n-3)n/9+(n-6)\binom{n}{2}/9$, which is also less than $\binom{n}{3}$.
It shows the number $x_{\vec{\lambda}_{\texttt{k}}}=(\binom{n}{k}-\gamma_{n,k})/(n/k)$ of the type $\vec{\lambda}_{\texttt{k}}$ needed is non-negative (it should be mentioned that the divisibility of $n/k$ to $\binom{n}{k}-\gamma_{n,k}$ is automatic by the setting). This proves the case when $n\ge k^2-k$.

For the remaining case $n=k^2-2k$,  we define the same $\vec{\lambda}_{\texttt{i}}$ and repeat the same assignment $x_{\vec{\lambda}_{\texttt{i}}}$ as the $n \ge k^2-k$ case for $i\in \{1,\cdots, k-3\}$.
We then let $x_{\vec{\lambda}}=\binom{n}{k-2}$ for the $(n,k)$-type $\vec{\lambda}$ with $\lambda_{k-2}=1$ and $\lambda_{k-1}=k-2$. As $\binom{n}{k-2}/\binom{n}{k-1}=1/(k-2)$,
these together take care of the first $k-1$ equations of the system \eqref{eq1.1}.
Finally, using the type $\vec{\lambda}_{\texttt{k}}$ and by the same estimation as the the $n \ge k^2-k$ case,
we can cover the level $k$ precisely.
This also gives a non-negative integer solution to the system \eqref{eq1.1}.
\end{proof}

\begin{lemma}\label{lem_solve2}
Suppose $n > 2k$ and $n=j\cdot k+k-1$, where $j$ is a non-negative such that $j \ge \lceil \frac{k}{2}\rceil -2$, then the system \eqref{eq1.1} has a non-negative integer solution.
\end{lemma}
\begin{proof}
Suppose $n=jk+(k-1)$ for $j \ge \lceil k/2 \rceil -2$.
As $n>2k$, in fact we have $n\ge 3k-1$.
First we consider the case when $k$ is even.
Then we have $n \ge k^2/2-k-1$ and $k$ divides $n+1$. Now instead of finding solutions to $A^T \vec{x}=b$,
we consider the $n+1$ case and the new system $(A')^T \vec{y}=\vec{b}'$.
Here $A'$ is the matrix defined for the $(n+1, \mathcal{L})$ case with $\mathcal{L}=\{2,4,\cdots,k\}$, and $\vec{b}'$ is a vector in $\mathbb{R}^k$ such that $b_i=\binom{n+1}{i}$ for even $i$, and $0$ for odd $i$. This system of linear equations has a non-negative integer solution by taking the type $\vec{\lambda}_{\texttt{i}}=(\lambda_1,\lambda_2,\cdots,\lambda_k)$ with
$$\lambda_i=\frac{k}{\gcd(k, i)},~~ \lambda_k=\frac{n+1}{k}-\frac{i}{\gcd(k, i)},~\mbox{ and }~\lambda_s=0~\textrm{for~}s \not\in \{i, k\}.$$
and assigning $\binom{n+1}{i}/(k/\gcd(k, i))$ to $y_{\vec{\lambda}_{\texttt{i}}}$, for $i=2, 4, \cdots, k-2$
and then using the type $\vec{\lambda}_{\texttt{k}}=(\lambda_1',\lambda_2',\cdots,\lambda_k')$ such that $\lambda_k'=(n+1)/k$ and $\lambda_s'=0$ for each $s\neq k$ to cover whatever is left for $b_k$.
Similar as before, all defined $y_{\vec{\lambda}}$ are integers. Furthermore, since we only do it for even $i$, $k$ is also even and $n+1\ge k(k-2)/2$, we always have
\begin{equation}\label{equ:gcd}
\frac{n+1}{k}-\frac{i}{\gcd(k, i)} \ge \frac{n+1}{k}-\frac{i}{2} \ge \frac{n+1}{k}-\frac{k-2}{2} \ge 0.
\end{equation}
It only remains to show that the type $\vec{\lambda}_{\texttt{k}}$ appears for a non-negative number of times.
This can be estimated as follows (in a similar way as in the proof of Lemma~\ref{lem_solve1})
\begin{align*}
 &\sum_{i\in \{2,4,\cdots,k-2\}} \frac{(n+1)/k-i/\gcd(k, i)}{k/\gcd(k, i)}\cdot\binom{n+1}{i}\\
= & \sum_{i\in \{2,4,\cdots,k-2\}} \left(\frac{n+1}{2k}-\frac{2}{k}\right)\cdot\binom{n+1}{i}\\
\le &\frac{n-3}{2k}\cdot \left(\left(\frac{k}{n-k+2}\right)^2+\left(\frac{k}{n-k+2}\right)^3+\cdots \right)\cdot \binom{n+1}{k}\\
< &\frac{n-3}{k}\cdot \left(\frac{k}{n-k+2}\right)^2\cdot\binom{n+1}{k}=\frac{k(n-3)}{(n-k+2)^2}\cdot \binom{n+1}{k}\le \binom{n+1}{k}=b_k,
\end{align*}
where we have $\binom{n+1}{k-i-1}/\binom{n+1}{k-i} \le k/(n-k+2)\le 1/2$ for $i\ge 0$ and $k(n-3)\le (n-k+2)^2$ under the fact that $n\geq 3k-1$.
So this indeed gives a non-negative integer solution to $(A')^T\vec{y}=\vec{b}'$.
Therefore by Theorem \ref{thm_reduction}, if we let $\mathcal{F}$ consist of all the subsets of $[n+1]$ of even size up to $k$, i.e., $\mathcal{F}=\binom{[n+1]}{\{2,4,\cdots,k\}}$, then this family can be decomposed into $1$-factors.
We delete the element $n+1$ from the unique subset containing it in each $1$-factor.
It is not hard to see that this immediately gives a $1$-factorization of $K_{n}^{\le k}$,
and thus the system $\eqref{eq1.1}$ has a non-negative integer solution (under the conditions of Lemma~\ref{lem_solve2}) when $k$ is even.

For the remaining proof, we assume that $k$ is odd.
In this case, we have $n=jk+k-1$ with $j \ge \lceil k/2 \rceil -2 = (k-3)/2$.
So we may assume that $$n=(k^2-k-2)/2+tk \mbox{~ for some integer } t \ge 0.$$
It is not hard to see that the above proof for the even $k$ case still works for the odd $k$ case (i.e., applied to $\binom{[n+1]}{\{1,3,\cdots,k-2,k\}}$)
whenever the corresponding form for \eqref{equ:gcd} holds.
However, for odd $k$ and $i\in \{1,3,\cdots,k-2\}$, we cannot ensure $\gcd(k,i)\geq 2$, which was used in \eqref{equ:gcd} for the even $k$ case.
As to infer that $\frac{n+1}{k}-\frac{i}{\gcd(k,i)}\ge \frac{n+1}{k}-i\ge 0$ holds for all $i\in \{1,3,\cdots,k-2\}$,
this approach requires an additional condition $n+1\geq k(k-2)$ or equivalently $t\ge (k-3)/2$.
Therefore, from now on we may further assume that $$0 \le t \le (k-5)/2.$$
Since $n=k(k+2t-1)/2-1$ and $(k+2t-1)/2\le k-3$,
if we view $(k+2t-1)/2$ as a new parameter $k'$, then $n=j'\cdot k'-1$ with $j'=k\ge k'+3$.
By the previous proofs, no matter if $k'$ is odd or even, we can always conclude that
$\binom{[n]}{\le (k+2t-1)/2}=\binom{[n]}{\le k'}$ is $1$-factorable.
So it suffices to show that the family consisting of subsets of $[n]$ of size between $(k+2t+1)/2$ and $k$,
denoted by $\binom{[n]}{\{(k+2t+1)/2, \cdots,k-1, k\}}$, is $1$-factorable.\footnote{Theorem~\ref{thm_reduction} states that for any set $\mathcal{L}$ of distinct positive integers,
the family $\binom{[n]}{\mathcal{L}}$ is $1$-factorable if and only if the corresponding system $\eqref{eq1.1}$ for $\mathcal{L}$ has a non-negative integer solution.
We are aware that we are looking for non-negative integer solutions for the (corresponding) system $\eqref{eq1.1}$, however for convenience of presentation, we should mention and identify both settings.}

The rest of the proof will be divided into two cases: when $t=(k-5)/2$ and when $0\leq t\leq (k-7)/2$.
First we consider the case when $t=(k-5)/2$.
In this case, we have $n=(k^2-k-2)/2+tk=k^2-3k-1$, and $(k+2t+1)/2=k-2$.
So it suffices to show $\binom{[n]}{\{k-2, k-1, k\}}$ is $1$-factorable.
We will use the following three $(n, \{k-2,k-1,k\})$-types (all unspecified coordinates $\lambda_i$ are $0$ by default):
%$(\lambda_1,\lambda_2,\cdots,\lambda_k)$, denoted by $\textsf{A}, \textsf{B}$ and $\textsf{C}$ respectively:
$$\textsf{A}:~~~~\lambda_{k-2}=(k+1)/2,~~\text{and}~~\lambda_k=(k-5)/2.$$
$$\textsf{B}:~~~~\lambda_{k-2}=(k-1)/2,~~~~\lambda_{k-1}=2,~~\text{and}~~\lambda_k=(k-7)/2.$$
$$\textsf{C}:~~~~\lambda_{k-1}=1,~~\text{and}~~\lambda_k=k-4.$$
Note that $\binom{n}{k}/\binom{n}{k-1}=(n-k+1)/k=k-4$ and type $\textsf{C}$ has the same ratio for level $k-1$ and $k$.
So if we assume that type $\textsf{A}$ is used $a$ times and type $\textsf{B}$ is used $b$ times,
then we need to find a non-negative integer solution to the following system of equations:
\begin{equation*}
\frac{k+1}{2}a+\frac{k-1}{2}b=\binom{n}{k-2}=\binom{k^2-3k-1}{k-2},
\end{equation*}
\begin{equation*}
	\left(\frac{k-5}{2}a+\frac{k-7}{2}b\right)\bigg/{2b}=k-4.
\end{equation*}
Solving it gives that
$$a=\frac{(k-3)/{2}}{(k^2-3k-1)/{3}}\binom{k^2-3k-1}{k-2}~~\text{and}~~b=\frac{(k-5)/{2}}{k^2-3k-1}\binom{k^2-3k-1}{k-2}.$$
Note that $\frac{k^2-3k-1}{3}$ divides $\frac{k^2-3k-1}{\gcd(3, k-2)}$, and $\frac{k^2-3k-1}{\gcd(3, k-2)}=\frac{k^2-3k-1}{\gcd(k^2-3k-1, k-2)}$ divides $\binom{k^2-3k-1}{k-2}$;
also $k$ is odd, so we can see that $a$ is a non-negative integer.
For $b$, note that when $k \not\equiv 2 \pmod 3$, $k^2-3k-1=\frac{k^2-3k-1}{\gcd(3, k-2)}$ divides $\binom{k^2-3k-1}{k-2}$,
while when $k \equiv2 \pmod 3$ ($k$ is odd, so in fact $k\equiv5 \pmod 6$),
$(k^2-3k-1)/3$ divides $\binom{k^2-3k-1}{k-2}$ and $(k-5)/6$ is an integer.
Therefore $b$ is also an non-negative integer.
Finally by considering the level $k-1$, we can use the type $\textsf{C}$ for $c$ times, where
$$c=\binom{n}{k-1}-2b= \binom{n}{k-1}-\frac{k-5}{k^2-3k-1}\binom{n}{k-2} \ge 0.$$
The above analysis gives a non-negative integer solution to the corresponding system $\eqref{eq1.1}$ for $\mathcal{L}=\{k-2,k-1,k\}$
and thus the proof for the case $t=(k-5)/2$ is complete.

Now we consider the last case when $0 \le t\le (k-7)/2$.
For this case, we have $k\geq 7$ and we want to show that $\binom{[n]}{\mathcal{L}}$ is $1$-factorable for $\mathcal{L}=\{(k+2t+1)/2, \cdots,k-1, k\}$.
To show that, we will use the following $(n, \mathcal{L})$-types labelled by $$\textsf{R}, \textsf{S}_0, \textsf{S}_1, \textsf{S}_2, \cdots, \textsf{S}_{\frac{k-5}{2}-t}, \textsf{T}_1, \textsf{T}_2, \cdots, \textsf{T}_{\frac{k-5}{2}-t}.$$
Each of them is a vector $\vec{\lambda}=(\lambda_1, \cdots, \lambda_k)$ defined as follows (here all the unspecified coordinates $\lambda_i$ are equal to $0$):
for each $1\le i\le (k-5)/2-t$,
$$\textsf{R}:~~~~\lambda_{k-1}=1,~~\text{and}~~\lambda_k=(k-3)/{2}+t,$$
$$\textsf{S}_0:~~~~\lambda_{k-2}=(k+1)/{2},~~\text{and}~~\lambda_k=t,$$
$$\textsf{S}_i:~~~~\lambda_{k-2-i}=1,~~~\lambda_{k-2}=(k-1)/{2}-i,~~~\lambda_{k-1}=i,~~\text{and}~~\lambda_k=t, \mbox{ and}$$
$$\textsf{T}_i:~~~~\lambda_{k-2-i}=2,~~~\lambda_{k-2}=(k-3)/{2}-i,~~~\lambda_{k-1}=0,~~\text{and}~~\lambda_k=t+i.$$
It is not hard to verify that each such $\vec{\lambda}$ satisfies $\sum_{\ell\in \mathcal{L}} \ell \lambda_{\ell}=n$ so all of them are indeed $(n,\mathcal{L})$-types.
The assumption $t\le (k-7)/2$ guarantees that some types other than $\textsf{R}$ and $\textsf{S}_0$ are used.
Let us assume that we use type $\textsf{R}$ for $x$ times, type $\textsf{S}_i$ for $a_i$ times for each $i\in \{0, 1, \cdots, (k-5)/2-t\}$,
and type $\textsf{T}_i$ for $b_i$ times for each $i\in \{1, \cdots, (k-5)/2-t\}.$
Note that only the type $\textsf{R}$ corresponds to 1-factors of size $(k-1)/{2}+t$,
while the other types correspond to 1-factors of size $(k+1)/{2}+t$.
If we let $y=a_0+a_1\cdots + a_{(k-5)/2-t}+b_1+\cdots + b_{(k-5)/2-t}$, then we have
\begin{equation}\label{eq_subset}
	\left(\frac{k-1}{2}+t\right)x+\left(\frac{k+1}{2}+t\right)y=\binom{n}{(k+1)/2+t}+\cdots + \binom{n}{k},
\end{equation}
and
\begin{equation}\label{eq_factor}
	x+y=\binom{n-1}{(k-1)/2+t}+\cdots + \binom{n-1}{k-1}.
\end{equation}
The first equation is derived from double counting the total number of subsets in $\binom{[n]}{\{(k+2t+1)/2, \cdots, k-1, k\}}$,
while the second equation follows by double counting the number of $1$-factors in the $1$-factorization of $\binom{[n]}{\{(k+2t+1)/2, \cdots, k-1, k\}}$.
Solving \eqref{eq_subset} and \eqref{eq_factor} (see Appendix \ref{App} for a detailed proof),
we can obtain the precise values of $x$ and $y$ and show that both $x$ and $y$ are non-negative integers,
where
\begin{equation}\label{eq_y2}
a_0+\sum_{i=1}^{(k-5)/2-t} (a_i+b_i)=y=\sum_{i=0}^{(k-5)/2-t} \frac{k-2+2t+i((k-1)/2+t)}{n} \binom{n}{k-2-i}.
\end{equation}

After fixing the number of times $x$ for type $\textsf{R}$ and the total number of times $y$ for all other types,
our next step is to use the types $\textsf{S}_1, \cdots, \textsf{S}_{(k-5)/2-t}, \textsf{T}_1, \cdots, \textsf{T}_{(k-5)/2-t}$ to fully occupy the levels from $(k+2t+1)/2$ to $k-3$. This give rise to the following equations:
\begin{equation}\label{eq_lower_levels}
	a_i+2b_i = \binom{n}{k-2-i}~~~\textrm{for each~}i\in\{1, \cdots, (k-5)/2-t\}.
\end{equation}
Also note that $\binom{n}{k}/\binom{n}{k-1}=(k-3)/2+t$ and the type $\textsf{R}$ maintains this ratio.
Hence, we also need to guarantee that the contributions of other types except $\textsf{R}$ to level $k-1$ and $k$ are at an $1: (\frac{k-3}{2}+t)$ ratio.
This leads to the following equality

\begin{equation}\label{eq_correct_ratio}
	ty+\sum_{i=1}^{(k-5)/2-t} ib_i-\left(\frac{k-3}{2}+t\right)\cdot\sum_{i=1}^{(k-5)/2-t} ia_i=0.
\end{equation}

It turns out that to find a non-negative integer solution to the corresponding system $\eqref{eq1.1}$ for $\mathcal{L}=\{(k+2t+1)/2,\cdots,k-1,k\}$,
it will suffice to find a non-negative integer solution $a_0, a_1, \cdots, a_{(k-5)/2-t}$, $b_1, \cdots, b_{(k-5)/2-t}$ to the system of $(k-1)/2-t$ equations formed by \eqref{eq_y2}, \eqref{eq_lower_levels} and \eqref{eq_correct_ratio}.
%(we note that the value of $x$ has been determined from \eqref{eq_subset} and \eqref{eq_factor}).
To solve the latter system, we could simply take
$$b_i=\left\lfloor\binom{n}{k-2-i}/2\right\rfloor ~\mbox{ and }~a_i=\binom{n}{k-2-i}-2b_i~~\textrm{for each~}i=2, \cdots, (k-5)/2-t,$$
and leave the other three variables $a_0, a_1, b_1$ to be uniquely determined by the three equations \eqref{eq_y2}, \eqref{eq_lower_levels} for $i=1$ and \eqref{eq_correct_ratio}.
A proof for showing the above assertions will be fully provided in Appendix \ref{App}.
We have now completed the proof of Lemma~\ref{lem_solve2}.
\end{proof}

\subsection{Proofs of Theorems \ref{thm_main} and \ref{thm_equiv}}
With all the preparations above, finally we are ready to address our main theorems.\medskip

\noindent {\bf Proof of Theorem \ref{thm_main}:} For $k<n/2$, suppose $n, k$ satisfy one of the two conditions. By Lemmas \ref{lem_solve1} and \ref{lem_solve2}, the system \eqref{eq1.1} has a non-negative integer solution. Theorem \ref{thm_reduction} immediately tells us that for such $n, k$, $K_{n }^{\le k}$ is $1$-factorable.

Now suppose $K_{n}^{\le k}$ is $1$-factorable. Again using Theorem \ref{thm_reduction}, it is neccesary that the system \eqref{eq1.1} has a non-negative integer solution. By Lemma \ref{th1.3}, $n$ must be congruent to $0$ or $-1$ mod $k$. Now apply Lemma \ref{th1.4}, we know that one of the two conditions in the statement of Theorem \ref{thm_main} must be met and this completes our proof. \qed
~\\

As of now, we have completely characterized all the $n, k$ in the range $k<n/2$ such that $K_n^{\le k}$ is $1$-factorable. The range $k \ge n/2$ could be tackled in a similar fashion by applying Farkas' Lemma and Theorem \ref{thm_reduction}. However, the statement of Theorem \ref{thm_equiv} already suggests that there is a very simple reduction to the $k <n/2$ range, as demonstrated below.

~\\
\noindent {\bf Proof of Theorem \ref{thm_equiv}:} We first show that for $k \ge n/2$, if $K_n^{\le n-k-1}$ is $1$-factorable, then $K_n^{\le k}$ is also $1$-factorable. Note that in this range, we always have $n-k \le k$. Take an arbitrary $1$-factorization $M_1, \cdots, M_t$ of $K_n^{\le n-k-1}$. For every subset $S$ of $[n]$ of size between $n-k$ and $k$, we just pair $S$ with its complement. This gives $\frac{1}{2}\sum_{i=n-k}^k \binom{n}{i}$ $1$-factors, which together with $M_1, \cdots, M_t$ form a $1$-factorization of $K_n^{\le k}$.

Next we prove the opposite direction. Suppose $k \ge n/2$ and $K_{n}^{\le k}$ can be decomposed into $1$-factors $M_1, \cdots, M_s$. For every $k$-set $S$, it must appear in some $M_i$. Suppose $M_i$ consists of the subset $S$, together with $\ell$ other subsets $T_1, \cdots, T_\ell$. Then $|T_1 \cup \cdots \cup T_\ell|=n-|S|=n-k \le k$. So $T_1 \cup \cdots \cup T_\ell$ must also appear in some $M_j$ (possibly $j=i$ if $\ell=1$). We move those $T_1, \cdots, T_{\ell}$ from $M_i$ to $M_j$, and also move $T_1 \cup \cdots \cup T_{\ell}$ from $M_j$ to $M_i$, to obtain a new $1$-factorization of $K_n^{\le k}$. Now $S$ is paired with $\overline{S}$. We repeat this process for every $k$-set $S$, and apply the same operation for sets of size between $n/2$ and $k-1$ as well. At the end of this process, we end up with a $1$-factorization of $K_{n}^{\le k}$ such that for each $n-k \le i \le k$, every $i$-set is paired with its complement. Removing these $1$-factors gives an $1$-factorization of $K_n^{\le n-k-1}$. \qed

\section{Extensions to other unions of levels of hypercube}\label{sec_multilayer}
Given a set $\mathcal{L}$ of distinct positive integers, recall that we denote by $\binom{[n]}{\mathcal{L}}$ the family of subsets of $[n]$ whose size is an element of $\mathcal{L}$. Can we find a necessary and sufficient condition for $\binom{[n]}{\mathcal{L}}$ to be $1$-factorable, when $n$ is sufficiently large? Our Theorem \ref{thm_main} answers this question for $\mathcal{L}=\{1, \cdots, k\}$, showing that the condition needed is simply $n \equiv 0, -1 \pmod k$. Does there exist such a neat sufficient and necessary condition for general $\mathcal{L}$? In this section we establish a number of results in this direction. For the proofs below, we always let $k$ be the maximum element of $\mathcal{L}$.
%Moreover, since we no longer insist on using every subset of size up to $k$, as a result in Equation \ref{eq1.1}, the rows of $A$ are formed only by those $(n, k)$-types $\vec{\lambda}$ with $\lambda_i=0$ for $i \not\in \mathcal{L}$, and the subset-counting vector $\vec{b}$ satisfies $b_i=\binom{n}{i}$ for $i \in \mathcal{L}$ and $b_i=0$ otherwise. It is not hard to see that Theorem \ref{thm_reduction} can still be adapted to this scenario.

\begin{theorem}\label{thm_general_0modk}
When $n$ is sufficiently large and divisible by $k$, $\binom{[n]}{\mathcal{L}}$ is always $1$-factorable.
\end{theorem}
\begin{proof}
Suppose $\mathcal{L}=\{\ell_1, \cdots, \ell_t\}$ with $\ell_1 < \cdots < \ell_t=k$, $n=jk$ and $j$ is sufficient large. We present a non-negative integer solution to the system \eqref{eq1.1} with $b_i=\binom{n}{i}$ for $i \in \mathcal{L}$ and $0$ otherwise, along the same line of the proof of Lemma \ref{lem_solve1}, and the conclusion follows from Theorem \ref{thm_reduction}.

For each $1 \le i \le t-1$, we take $\vec{\lambda} \in \mathbb{Z}_{\ge 0}^k$ such that
$$\lambda_{\ell_i}=\frac{k}{\gcd(k, \ell_i)},~~ \lambda_k=\frac{n}{k}-\frac{\ell_i}{\gcd(k, \ell_i)},~~\lambda_s=0~\textrm{for~}s \not\in \{\ell_i, k\}.$$
Similarly as before, it is not hard to see that $\vec{\lambda}$ is a $(n, \mathcal{L})$-type for sufficiently large $n$ divisible by $k$. Now we just let $x_{\vec{\lambda}}=\binom{n}{\ell_i}/(k/\gcd(k, \ell_i))$ to take care of $b_{\ell_i}$ for each $i=1, \cdots, t-1$. Finally we use the type $\vec{\lambda}=(0, \cdots, 0, n/k)$ to take care of the remainder for $b_k$. The number of such types is non-negative follows similarly as before.
\end{proof}

Our next result generalizes Lemma \ref{th1.3}.

\begin{theorem}
For sufficiently large $n$, if $\binom{[n]}{\mathcal{L}}$ is $1$-factorable, then $n \equiv 0$ or $-1 \pmod k$.

\end{theorem}
\begin{proof}
Suppose $n=jk+r$ for $1 \le r \le k-2$. We use Farkas' Lemma and take $\vec{y} \in \mathbb{R}^k$ such that
$y_k=-1$, $y_r=j$ and $y_s=j/2$ for $s \not\in \{k, r\}$. First we explain why $A_{\mathcal{L}} \vec{y} \ge \vec{0}$. Take a row $\vec{\lambda}$ of $A_{\mathcal{L}}$, by our definition of $A_{\mathcal{L}}$, we have $\sum_{i=1}^k i\lambda_i=n$ and $\lambda_i=0$ for those $i \not\in \mathcal{L}$. Note that $n=jk+r < (j+1)k$, we have $\lambda_k \le j$ and $\sum_{i=1}^{k-1} i\lambda_i = r$ or $\sum_{i=1}^{k-1} i\lambda_i \ge r+k$. These imply either $\lambda_r=1$ or $\sum_{i=1}^{k-1} \lambda_i \ge 2$. In either case we have
$$\sum_{i=1}^k \lambda_i y_i = \left(\sum_{i=1}^{k-1} \lambda_k y_k \right) + \lambda_k y_k \ge j-j = 0.$$
Next we show that $\vec{y}\cdot \vec{b}_{\mathcal{L}} < 0$. Equivalently, we need to prove
$$\sum_{i \in \mathcal{L}\setminus \{k\}} \frac{j}{2} \binom{n}{i} + {\bf 1}_{r \in \mathcal{L}} \cdot \frac{j}{2} \cdot \binom{n}{r}<\binom{n}{k}.$$
We instead prove a stronger inequality
\begin{equation}\label{ineq_mod}
\sum_{i=1}^{k-1} \frac{j}{2} \binom{n}{i} +  \frac{j}{2} \cdot \binom{n}{r}<\binom{n}{k}.
\end{equation}
Note that for $i \ge 1$,
$$\binom{n}{k-i}/\binom{n}{k-i+1}=\frac{k-i+1}{n-k+i}=\frac{k-i+1}{(j-1)k+i+r} \le \frac{k}{(j-1)k+r+1} \le  \frac{1}{j-1}.$$
Also we have $\binom{n}{r} \le \binom{n}{k-2}$. Therefore the left hand side of \eqref{ineq_mod} is at most
\begin{align*}
&~~~\left(\frac{1}{j-1}+\left(\frac{1}{j-1}\right)^2+\cdots\right)\cdot\frac{j}{2}\cdot\binom{n}{k} + \frac{j}{2} \cdot \frac{1}{(j-1)^2}\cdot \binom{n}{k}\\
&<\left(\frac{j}{2(j-2)}+\frac{j}{2(j-1)^2}\right) \binom{n}{k} < \binom{n}{k}
\end{align*}
for $j \ge 5$. Since we assume that $n$ is sufficiently large, this completes the proof.
\end{proof}

When $n \equiv -1 \pmod k$, we can further prove the following.
\begin{theorem}
Suppose $n$ is sufficient large and $n=jk+k-1$ for some positive integer $k$. If $\binom{[n]}{\mathcal{L}}$ is $1$-factorable, then $\mathcal{L}$ must contain the element $k-1$.
\end{theorem}
\begin{proof}
We prove the contra-positive: suppose $k-1 \not\in \mathcal{L}=\{\ell_1, \cdots, \ell_t\}$ with $\ell_1< \cdots <\ell_t=k$, then $\binom{[n]}{\mathcal{L}}$ is not $1$-factorable. Once again we would like to use Farkas' Lemma to show that $A_{\mathcal{L}}^T\vec{x}=\vec{b}_{\mathcal{L}}$ has no non-negative solution. Take the vector $\vec{y} \in \mathbb{R}^k$ such that
$$y_{\ell_i} = \frac{j}{2}~~\textrm{for~}i=1, \cdots, t-1,~~~~~~~y_{k}=-1.$$
The rest of the coordinates of $\vec{y}$ are set to be zero. First we show that $A_{\mathcal{L}}\vec{y} \ge \vec{0}$. Take a row  $\vec{\lambda}=(\lambda_1, \cdots, \lambda_k)$, recall that $\lambda_i=0$ for $i \not\in \mathcal{L}$. So
$\sum_{i=1}^t \ell_i \lambda_{\ell_i}=n=jk+(k-1)$. Therefore $\lambda_k \le j$ and $\sum_{i=1}^{t-1} \lambda_{\ell_i} \ge 1$. If $\sum_{i=1}^{t-1} \lambda_{\ell_i} \ge 2$, then
$$\sum_{i=1}^k \lambda_i y_i \ge \frac{j}{2}\cdot \left(\sum_{i=1}^{t-1}\lambda_{\ell_i}\right)-\lambda_k \ge j-j=0.$$
If $\sum_{i=1}^{t-1} \lambda_{\ell_i}=1$, then $n$ is congruent to $\ell_i$ modulo $k$, for some $1 \le i \le t-1$, this is not possible since $k-1 \not\in \mathcal{L}$ and thus all such $\ell_i$ are at most $k-2$, but we have $n \equiv k-1 \pmod k$.

Next we prove $\vec{b}_{\mathcal{L}}\cdot\vec{y}<0$. Since $k-1 \not \in \mathcal{L}$, it suffices to prove
\begin{equation}\label{ineq_large}
\sum_{i=1}^{k-2} \frac{j}{2} \binom{n}{i}<\binom{n}{k}.
\end{equation}
Note that for every $i \ge 1$,
$$\binom{n}{k-i}/\binom{n}{k-i+1}=\frac{k-i+1}{n-k+i}=\frac{k-i+1}{jk+i-1} \le \frac{1}{j}.$$
Therefore the left hand side of inequality \eqref{ineq_large} is at most
$$\frac{j}{2}\left(\frac{1}{j^2}+\frac{1}{j^3}+\cdots\right) \binom{n}{k}< \frac{1}{2(j-1)}\binom{n}{k} \le \binom{n}{k},$$
as long as $j \ge 2$.
\end{proof}

We make the following conjecture based on the above results.
\begin{conjecture}\label{conj_general_L}
Given a set of distinct positive integers $\mathcal{L}$ whose largest element is $k$. Depending on the choice of $\mathcal{L}$, exactly one of the following statements must be true:\\
(i) For sufficiently large $n$, $\binom{[n]}{\mathcal{L}}$ is $1$-factorable if and only if $n \equiv 0 \pmod k$, \\
(ii) For sufficiently large $n$, $\binom{[n]}{\mathcal{L}}$ is $1$-factorable if and only if $n \equiv 0, -1 \pmod k$.
\end{conjecture}

When $k-1 \in \mathcal{L}$, it seems that both scenarios are possible.
For example, $\mathcal{L}=\{k-1, k\}$ satisfies $(ii)$, since for $n=jk-1$, we could use the type $\vec{\lambda}=(0, \cdots, 0, 1, j-1)$ exactly $\binom{n}{k-1}$ times, and Theorem \ref{thm_general_0modk} takes care of the $n=jk$ case.
On the other hand, if we take $k=4$ and $\mathcal{L}=\{2,3,4\}$, then for $n=4j-1$, applying Farkas' Lemma with the vector $\vec{y}=(-1/2, j-1, -1)^T$ shows that $\binom{[n]}{\mathcal{L}}$ is not $1$-factorable, and thus  $\mathcal{L}=\{2,3,4\}$ corresponds to case $(i)$ in Conjecture \ref{conj_general_L}. It would be great if some criteria on those $\mathcal{L}$ satisfying $k-1 \in \mathcal{L}$ could be found to tell whether $\binom{[n]}{\mathcal{L}}$ is $1$-factorable for sufficiently large $n \equiv -1 \pmod k$.

\section{Concluding Remarks}\label{sec_concluding}
In this paper, we determine all the pairs of positive integers $(n, k)$ such that $\binom{[n]}{\le k}$ is $1$-factorable. We include a few open questions and possible future projects below.

\begin{itemize}

\item Our method might be extendable to the scenario when repetition of subsets are allowed in the hypergraph. In \cite{bah}, Bahmanian show that the existence of a symmetric Latin cube is equivalent to the existence of a partition into parallel classes of some non-uniform set system. In particular, he determined all $n$ such that $\binom{[n]}{1} \cup 3 \binom{[n]}{2} \cup 2 \binom{[n]}{3}$ is $1$-factorable. Here every pair appears three times and every triple shows up twice. Would it be possible to determine for which $n$ and sequence $\{m_i\}_{i=1, \cdots, n}$ of non-negative integers, $\cup_{i=1}^n m_i \binom{[n]}{i}$ is $1$-factorable?

\item A famous conjecture of Chv\'atal \cite{chvatal_conj} says that any for any given {\it hereditary family} $\mathcal{F}$ (i.e. $A \in \mathcal{F}$ and $B \subset A$ implies $B \in \mathcal{F}$), the maximum size of its intersecting subfamily is attained by taking all sets  containing the most popular element. If we construct a graph $G_{\mathcal{F}}$ such that $V(G_{\mathcal{F}})=\mathcal{F}$, and two vertices $S, T \in \mathcal{F}$ are adjacent if and only if $S \cap T =\emptyset$, then Chv\'atal's conjecture is equivalent to showing $\alpha(G_{\mathcal{F}})=\max_x |\mathcal{F}_x|$, here we denote by $\mathcal{F}_x$ the family of subsets containing $x$. Observe that when $\mathcal{F}$ is {\it regular}, i.e. every element in the ground set occurs for an equal number of times, $\mathcal{F}$ is $1$-factorable if and only if $\chi(\overline{G_{\mathcal{F}}})=\max_x |\mathcal{F}_x|$, which is a stronger bound and immediately implies $\alpha(G_{\mathcal{F}})=\max_x |\mathcal{F}_x|$. Inspired by Theorem \ref{thm_main}, one may wonder whether there is a nice characterization of all the regular hereditary families $\mathcal{F}$ that are $1$-factorable.

\end{itemize}

\appendices

\section{Finding a non-negative integer solution in the proof of Lemma~\ref{lem_solve2}}\label{App}
We consider the case when $k$ is odd, $n=(k^2-k-2)/2+tk$ and $0 \le t\le (k-7)/2.$
In particular $k\geq 7$.
Our aim here is to provide detailed arguments to the proof described for this case of Lemma~\ref{lem_solve2},
as to find a non-negative integer solution to the corresponding system $\eqref{eq1.1}$ for $\mathcal{L}=\{(k+2t+1)/2,\cdots,k-1,k\}$.

First let us determine $x$ and $y$ as non-negative integers.
From \eqref{eq_subset} and \eqref{eq_factor},
\begin{align}\label{eq_y}
	y&=\sum_{i=(k+1)/2+t}^{k}\left(\binom{n}{i}-\left(\frac{k-1}{2}+t\right) \cdot\binom{n-1}{i-1}\right)\nonumber \\
	&=\left(\sum_{i=(k+1)/2+t}^{k-2} \frac{n-i((k-1)/2+t)}{n}\cdot \binom{n}{i}\right) + \frac{(k-3)/2+t}{n} \binom{n}{k-1}  -\frac{1}{n}\binom{n}{k}\nonumber\\
	&=\sum_{i=(k+1)/2+t}^{k-2} \frac{n-i((k-1)/2+t)}{n}\cdot \binom{n}{i}\nonumber\\
	&=\sum_{i=0}^{(k-5)/2-t} \frac{k-2+2t+i((k-1)/2+t)}{n} \binom{n}{k-2-i}.
\end{align}
From the first equality, we know that $y$ is an integer, and the last form clearly indicates that $y$ is non-negative.
From \eqref{eq_subset} and \eqref{eq_factor}, using $n=k(\frac{k+1}{2}+t)-(k+1)$ and $\frac{k+1}{2}+t\leq k+1$, we also have
\begin{align*}\label{eq_x}
	x&=\sum_{i=(k+1)/2+t}^{k}\left(\left(\frac{k+1}{2}+t\right) \cdot\binom{n-1}{i-1}-\binom{n}{i}\right)\\
	&=\sum_{i=(k+1)/2+t}^{k} \frac{i((k+1)/2+t)-n}{n}\cdot \binom{n}{i}\\
	&=\frac{1}{n}\sum_{i=(k+1)/2+t}^{k} \left((k+1)-(k-i)(\frac{k+1}{2}+t)\right)\cdot\binom{n}{i}\\
    &\ge\frac{1}{n}\sum_{i=(k+1)/2+t}^{k-1} \left((k+1)\cdot\binom{n}{i+1}-(k-i)(\frac{k+1}{2}+t)\cdot\binom{n}{i}\right)\\
    &\ge\frac{k+1}{n}\sum_{i=(k+1)/2+t}^{k-1} \frac{n-i-(k-i)(i+1)}{i+1}\cdot \binom{n}{i}\ge 0,
\end{align*}
where the first equality shows $x$ is an integer and the last inequality follows from the fact that $n-i-(k-i)(i+1)=n-(k-i)i-k\geq \frac{k^2-k-2}{2}-\frac{k^2}{4}-k\geq 0$ holds for $k\geq 7$.

Recall the system of $(k-1)/2-t$ equations formed by \eqref{eq_y2}, \eqref{eq_lower_levels} and \eqref{eq_correct_ratio},
and that to find a non-negative integer solution $a_0, a_1, \cdots, a_{(k-5)/2-t}$, $b_1, \cdots, b_{(k-5)/2-t}$ to this system,
we take for each $i\in \{2, 3,\cdots, (k-5)/2-t\}$,
\begin{equation}\label{eq_abi}
b_i=\left\lfloor\binom{n}{k-2-i}/2\right\rfloor ~\mbox{ and }~a_i=\binom{n}{k-2-i}-2b_i\in \{0,1\}.
\end{equation}
It remains to find non-negative integers $a_0, a_1, b_1$ for the system.
Consider $A={\sum_{i=1}^{(k-5)/2-t} ia_i}$ and $B={\sum_{i=1}^{(k-5)/2-t} ib_i}$. Rewriting \eqref{eq_correct_ratio}, we obtain
\begin{equation}\label{B_in_A}
B=\left(\frac{k-3}{2}+t\right)A-ty.
\end{equation}
Now we multiply $i$ to \eqref{eq_lower_levels}, and sum over all defined $i$, we have
\begin{equation}\label{eq_A+2B}
A+2B=\sum_{i=1}^{(k-5)/2-t} i \binom{n}{k-2-i}.
\end{equation}
Plugging in the expression \eqref{B_in_A} on $B$ to \eqref{eq_A+2B}, we can solve $A$ as follows:
\begin{align} \label{eq_A}
	&(k-2+2t)A=2ty+\sum_{i=1}^{(k-5)/2-t} i \binom{n}{k-2-i}\nonumber\\
	=~&\sum_{i=0}^{(k-5)/2-t} \frac{(k-2+2t)(2t+i(t+\frac{k+1}{2})))}{n}\cdot \binom{n}{k-2-i} \\
	=~&(k-2+2t)\sum_{i=0}^{(k-5)/2-t} \left(2t+i(t+\frac{k+1}{2})\right) \frac{\binom{n}{k-2-i}}{n}\nonumber\\
	=~&(k-2+2t)\sum_{i=0}^{(k-5)/2-t} \left(n-(k-2-i)(\frac{k+1}{2}+t)\right) \frac{\binom{n}{k-2-i}}{n}\nonumber\\
	=~&(k-2+2t)\sum_{i=0}^{(k-5)/2-t} \left(\binom{n}{k-2-i}-(\frac{k+1}{2}+t)\binom{n-1}{k-3-i}\right).\nonumber
\end{align}
This shows that $A$ is an non-negative integer (where the third last form shows the non-negativity).
Using \eqref{B_in_A}, $B$ is also an integer.
%We will prove the non-negativity of $B$ later.
%Recall that $B={\sum_{i=1}^{(k-5)/2-t} ib_i}$. We just choose $b_i$ for $i$ in descending order, such that $b_i$ is as large as possible, while satisfying $ib_i \le B- (\sum_{j=i+1}^{(k-5)/2-t} j b_j)$ and $2b_i \le \binom{n}{k-2-i}$. It is not hard to see that one could take all $b_i$ to be non-negative integers. Plugging those $b_i$ back to \eqref{eq_lower_levels} gives the value of $a_i$.

We now show $a_1, b_1, a_0$ are non-negative integers in order.
First we determine $a_1$ from $A$. If $\frac{k-5}{2}-t=1$, then clearly $a_1=A$ is a non-negative integer.
Assume $\frac{k-5}{2}-t\ge 2$, which implies that $k\ge 9$.
By \eqref{eq_abi}, we have $a_i\in\{0, 1\}$ for $2\le i\le (k-5)/2-t$.
Using \eqref{eq_A} that $A=\sum_{i=0}^{(k-5)/2-t}\frac{2t+i(t+(k+1)/2)}{n}\binom{n}{k-2-i}$, we then get
\begin{align*}
    a_1&=A-\sum_{i=2}^{(k-5)/2-t}ia_i= \sum_{i=2}^{(k-5)/2-t}\left(2t+i(t+\frac{k+1}{2})\right)\frac{\binom{n}{k-2-i}}{n}-\sum_{i=2}^{(k-5)/2-t}ia_i\\
    &\ge \sum_{i=2}^{(k-5)/2-t}\left[\left(2t+i(t+\frac{k+1}{2})\right)\frac{\binom{n}{k-2-i}}{n}-i\right]\ge \sum_{i=2}^{(k-5)/2-t}\left(\frac{k+1}{2}-1\right)i\ge 0,
\end{align*}
as desired.
Since $B$ is an integer, it is easy to see that $b_1=B-\sum_{i=2}^{(k-5)/2-t}ib_i$ is also an integer.
We also have $a_1+2b_1=\binom{n}{k-2-1}$ from \eqref{eq_lower_levels}, which gives
\begin{align*}
    b_1&=\frac{1}{2}\left(\binom{n}{k-3}-a_1\right)
    \ge\frac{1}{2}\left(\binom{n}{k-3}-A\right).
\end{align*}
Using \eqref{eq_A}, $\binom{n}{k-2}/\binom{n}{k-3}=\frac{n-k+3}{k-2}$,
and $\binom{n}{k-3-i-1}/\binom{n}{k-3-i}\le \frac{k-3}{n-k+4}$ for $i\ge0$,
we have
\begin{align*}
    \binom{n}{k-3}-A&=\left(1-\frac{2t(n-k+3)}{n(k-2)}-\frac{\frac{k+1}{2}+3t}{n}\right)\binom{n}{k-3}-\sum_{i=2}^{(k-5)/2-t}\frac{2t+\left(\frac{k+1}{2}+t\right)i}{n}\binom{n}{k-2-i}\\
    &=\frac{k-3-2t}{k-2}\binom{n}{k-3}-\sum_{i=2}^{(k-5)/2-t}\frac{2t+\left(\frac{k+1}{2}+t\right)i}{n}\binom{n}{k-2-i}\\
    &\ge\frac{k-3-2t}{k-2}\binom{n}{k-3}-\sum_{i=2}^{(k-5)/2-t}\frac{(k^2-4k-5)/4}{n}\binom{n}{k-2-i}\\
    &\ge\frac{k-3-2t}{k-2}\binom{n}{k-3}-\frac{k^2-4k-5}{4n}\left(\frac{k-3}{n-k+4}+\left(\frac{k-3}{n-k+4}\right)^2+\cdots\right)\binom{n}{k-3}\\
    &\ge\frac{k-3-2t}{k-2}\binom{n}{k-3}-\frac{k^2-4k-5}{4n}\cdot\frac{k-3}{n-2k+7}\binom{n}{k-3}\\
    &\ge\left(\frac{4}{k-2}-\frac{1}{k-2}\cdot\frac{k^2-8k+15}{k^2-5k+12}\right)\binom{n}{k-3}\ge \frac{3}{k-2}\binom{n}{k-3}\ge 0,
\end{align*}
where the third last inequality uses that $t \le (k-7)/2$ and $n \ge (k^2-k-2)/2$. From the analysis above, we have shown that $b_1$ is a non-negative integer.

%It is not hard to see that
%\begin{align*}
%\sum_{i=2}^{(k-5)/2-t}ib_i \le \frac{1}{2} \sum_{i=2}^{(k-5)/2-t} i \binom{n}{k-2-i}
%\end{align*}

Recall that $y=a_0+\sum_{i=1}^{(k-5)/2-t}(a_i+b_i)$.
It is easy to check that $\sum_{i=1}^{(k-5)/2-t}a_i \le\sum_{i=1}^{(k-5)/2-t}ia_i=A$ and $b_i\le \frac{1}{2}\binom{n}{k-2-i}$ for $2\le i\le(k-5)/2-t$.
In addition, we have $b_1=\frac{1}{2}\left(\binom{n}{k-3}-a_1\right)$ and $a_1\ge A-\sum_{i=2}^{(k-5)/2-t}i$. Thus we have
\begin{align*}
    a_0&=y-\sum_{i=1}^{(k-5)/2-t}a_i-\sum_{i=1}^{(k-5)/2-t}b_i\\
    &\ge y-A-\frac{1}{2}\sum_{i=2}^{(k-5)/2-t}\binom{n}{k-2-i}-\frac{1}{2}\left(\binom{n}{k-3}-\left(A-\sum_{i=2}^{(k-5)/2-t}i\right)\right)\\
    &= y-\frac{1}{2}A-\frac{1}{2}\sum_{i=1}^{(k-5)/2-t}\binom{n}{k-2-i}-\frac{\left((k-1)/2-t\right)\left((k-7)/2-t\right)}{4}.
\end{align*}
Recall from \eqref{eq_y} that
\begin{equation}\label{eq_yy}
y=\sum_{i=0}^{(k-5)/2-t}\frac{k-2+2t+i(t+(k-1)/2)}{n}\binom{n}{k-2-i}
\end{equation}
and from \eqref{eq_A} that $A=\sum_{i=0}^{(k-5)/2-t}\frac{2t+i(t+(k+1)/2)}{n}\binom{n}{k-2-i}$.
Using $((k-1)/2-t)((k-7)/2-t)\le (k-1)(k-7)/4\le n/2$ and $\binom{n}{k-2-i-1}/\binom{n}{k-2-i}\le \frac{k-2}{n-k+3}$ for $i\ge 0$, we can derive that (note that $2k-4-n<0$)
\begin{align*}
    a_0&\ge \left(\sum_{i=1}^{(k-5)/2-t}\frac{2k-4+2t+i\left(\frac{k-3}{2}+t\right)-n}{2n}\binom{n}{k-2-i}\right)+\frac{k-2+t}{n}\binom{n}{k-2}-\frac{n}{8}\\
    &\ge \left(\sum_{i=1}^{(k-5)/2-t}\frac{2k-4-n}{2n}\binom{n}{k-2-i}\right)+\frac{k-2}{n}\binom{n}{k-2}-\frac{n}{8}\\
    &\ge -\left(\frac{n-2k+4}{2n}\right)\left(\frac{k-2}{n-k+3}+\left(\frac{k-2}{n-k+3}\right)^2\cdots\right)\binom{n}{k-2}+\frac{k-2}{n}\binom{n}{k-2}-\frac{n}{8}\\
    &=\frac{k-2}{n}\binom{n}{k-2}-\frac{n-2k+4}{2n}\frac{k-2}{n-2k+5}\binom{n}{k-2}-\frac{n}{8}\\
    &\ge\frac{k-2}{2n}\binom{n}{k-2}-\frac{n}{8}\ge\frac{1}{2}\binom{n-1}{k-3}-\frac{n}{8}\ge 0.
\end{align*}
This completes the proof for finding a non-negative integer solution $$\mathcal{S}=\{a_0, a_1, \cdots, a_{(k-5)/2-t}, b_1, \cdots, b_{(k-5)/2-t}\}$$ to the system formed by \eqref{eq_y2}, \eqref{eq_lower_levels} and \eqref{eq_correct_ratio}.

Lastly, we show that the above non-negative integer solution $\mathcal{S}$ together with the non-negative integer $x$ solved from \eqref{eq_subset} and \eqref{eq_factor},
along with the corresponding $(n,\mathcal{L})$-types, give a non-negative integer solution to the corresponding system $\eqref{eq1.1}$ for $\mathcal{L}=\{(k+2t+1)/2,\cdots,k-1,k\}$.
By \eqref{eq_lower_levels}, the contribution of these types to level $k-2-i$ satisfies
\begin{equation}\label{eq_a+b}
a_i+2b_i = \binom{n}{k-2-i} \mbox {~~for each ~} i\in\{1, \cdots, (k-5)/2-t\}.
\end{equation}
Thus, it is enough to check that the contribution of these types to each level $j\in \{k-2, k-1, k\}$ equals $\binom{n}{j}$.
We observe that as $\mathcal{S}$ satisfies \eqref{eq_subset}, \eqref{eq_factor} and \eqref{eq_correct_ratio},
it in fact suffices to verify that the contribution of these types to level $k-2$ equals $\binom{n}{k-2}$,
that is, we need to show the following
\begin{equation}\label{eq_level_k-2}
\frac{k+1}{2}a_0+\sum_{i=1}^{(k-5)/2-t} \left(\frac{k-1}{2}-i\right)a_i+\sum_{i=1}^{(k-5)/2-t} \left(\frac{k-3}{2}-i\right)b_i=\binom{n}{k-2}.
\end{equation}
We claim that the above equation \eqref{eq_level_k-2} is a linear combination of \eqref{eq_correct_ratio}, \eqref{eq_yy}, and \eqref{eq_a+b} for all $i\in\{1, \cdots, (k-5)/2-t\}$,
with coefficients $$\frac{1}{k-2+2t}, \frac{n}{k-2+2t}, \mbox{ and } -\frac{k-2+2t+i((k-1)/{2}+t)}{k-2+2t} \mbox{ for all } i\in\{1, \cdots, \frac{k-5}{2}-t\},$$ respectively.
It is easy to see that the right side hand of the above linear combination equals $\binom{n}{k-2}$.
For the left side hand of this linear combination, a careful calculation, using $t+n=\frac{k+1}{2}(k-2+2t)$ and $y=a_0+\sum_{i=1}^{(k-5)/2-t}(a_i+b_i)$, shows that it equals
\begin{align*}
&\frac{(t+n)y}{k-2+2t}+\frac{1}{k-2+2t} \sum_{i=1}^{(k-5)/2-t} \left[i-2\left(k-2+2t+i\big(\frac{k-1}{2}+t\big)\right)\right]b_i\\
&- \frac{1}{k-2+2t} \sum_{i=1}^{(k-5)/2-t}\left[\big(\frac{k-3}{2}+t\big)i+\left(k-2+2t+i\big(\frac{k-1}{2}+t\big)\right)\right]a_i\\
=& \frac{k+1}{2}\left(a_0+\sum_{i=1}^{(k-5)/2-t}(a_i+b_i)\right)-\sum_{i=1}^{(k-5)/2-t}(i+2)b_i-\sum_{i=1}^{(k-5)/2-t}(i+1)a_i\\
=& \frac{k+1}{2}a_0+\sum_{i=1}^{(k-5)/2-t} \left(\frac{k-1}{2}-i\right)a_i+\sum_{i=1}^{(k-5)/2-t} \left(\frac{k-3}{2}-i\right)b_i,
\end{align*}
which is the left side hand of \eqref{eq_level_k-2}. This proves our claim, completing the proof of this appendix.\qed
\end{document}